\RequirePackage{fix-cm}
\documentclass[smallextended]{svjour3}       
\smartqed  
\usepackage{graphicx}
\usepackage{algorithm,algorithmic}
\usepackage{amssymb}
\usepackage{amsmath}
\usepackage{mathdots}
\usepackage{bm}
\usepackage{booktabs}
\usepackage{multirow}
\usepackage[colorlinks,
            linkcolor=red,
            anchorcolor=blue,
            citecolor=blue
            ]{hyperref}
\usepackage{tabularx}
\usepackage{makecell}
\usepackage{graphicx}
\usepackage{subfigure}
\usepackage{lineno}
\usepackage[misc]{ifsym}
\graphicspath{{figure/}}
\DeclareGraphicsExtensions{.pdf}

\let\doendproof\endproof
\renewcommand\endproof{~\hfill$\qed$\doendproof}

\usepackage[final]{changes} 
\definechangesauthor[name={Author}, color=orange]{Author}
\begin{document}
\title{Fast gradient method for Low-Rank Matrix Estimation
}


\author{Hongyi Li \and Zhen Peng \and Chengwei Pan \and Di Zhao.}

\institute{H. Li, Z. Peng and D. Zhao  \at
	LMIB, School of Mathematical Sciences, Beihang University, Beijing 100191, China
	\and 
	C. Pan \at
	Institute of Artificial Intelligence, Beihang University, Beijing 100191, China
}
\date{Received: date / Accepted: date}

\maketitle

\begin{abstract}
Projected gradient descent and its Riemannian variant belong to a typical class of methods for low-rank matrix estimation.
This paper proposes a new Nesterov's Accelerated Riemannian Gradient algorithm using efficient orthographic retraction and tangent space projection.
The subspace relationship between iterative and extrapolated sequences on the low-rank matrix manifold provides computational convenience.
With perturbation analysis of truncated singular value decomposition and two retractions, we systematically analyze the local convergence of gradient algorithms and Nesterov's variants in the Euclidean and Riemannian settings.
Theoretically, we estimate the exact rate of local linear convergence under different parameters using the spectral radius in a closed form and give the optimal convergence rate and the corresponding momentum parameter.
When the parameter is unknown, the adaptive restart scheme can avoid the oscillation problem caused by high momentum, thus approaching the optimal convergence rate.
Extensive numerical experiments confirm the estimations of convergence rate and demonstrate that the proposed algorithm is competitive with first-order methods for matrix completion and matrix sensing.


\end{abstract}

\section{Introduction}
\label{sec:intro}

Recently, low-rank matrix estimation, as a fundamental model, has played an irreplaceable role in signal processing and machine learning~\cite{chi2019nonconvex}.
Such a model aims to recover complete information with a latent low-rank structure from the collected measurements $\boldsymbol y=\mathcal A(\boldsymbol X_\star)$, which is described as follows:
\begin{eqnarray}\label{eq:lowrank}
\min_{\boldsymbol X} f(\boldsymbol X):=\|\mathcal A(\boldsymbol X)-\boldsymbol y\|_2^2 \quad \text{s.t.}\quad \text{rank}(\boldsymbol X) \replaced{=}{\leq} r,
\end{eqnarray}
where $\mathcal A:\mathbb R^{n_1\times n_2}\mapsto \mathbb R^{m}$ is a linear operator, which arises in various applications, such as Matrix Completion (MC) and Matrix Sensing (MS).
Classical convex relaxation bypasses the computationally intractable nonconvex low-rank constraint by nuclear norm minimization.
Nonetheless, the computational and space complexity proportional to the matrix size severely limits the applicability of convex relaxations to large-scale problems.
Therefore, the nonconvex optimization of the model (\ref{eq:lowrank}) attracts more attention from researchers~\cite{chen2018harnessing}.
As a well-known class of low-rank matrix estimation algorithms, projected gradient descent alternates between vanilla gradient descent and low-rank matrix projection~\cite{davenport2016overview}.
The typical one performs a hard-thresholding operation on singular values, thus termed Iterative Hard Thresholding (IHT)~\cite{jain2010guaranteed}.
As these Euclidean methods suffer from a high computational burden associated with truncated Singular Value Decomposition (SVD), growing attention has turned to Riemannian optimization~\cite{vandereycken2013low}.
The fact that the rank of the tangent vector does not exceed $2r$ provides an efficient implementation of truncated SVD~\cite{cai2018exploiting}, which inspires a large class of Riemannian gradient descent (RGrad) algorithms.

Since the Heavy-ball method~\cite{polyak1964some} and Nesterov's Accelerated Gradient (NAG) method~\cite{nesterov1983method}, the introduction of momentum is one of the conventional ways to overcome the short-sighted issue of the gradient algorithm~\cite{li2019accelerated,wang2022stochastic}.
Theoretically, the NAG algorithm with optimal parameters can match the lower bound of the first-order optimization algorithm~\cite{li2020accelerated}.
However, on the one hand, optimal parameters are often challenging in practice.
On the other hand, iterative sequence oscillations caused by inappropriate parameters can significantly degrade performance.
To address the parameter selection issue, a seminal adaptive restart scheme~\cite{o2015adaptive} resets momentum when extrapolation is in the wrong direction.
Specifically for MC, Vu et al.~\cite{vu2019accelerating} accurately estimate local linear convergence of a NAG version of IHT via spectral radius, which has recently been generalized under general constraints~\cite{vu2021asymptotic}.
The adaptive restart scheme verifies the optimal asymptotic convergence rate in numerical results.
Nevertheless, this Euclidean-based acceleration does not enjoy the advantage of the excellent tools on the low-rank matrix manifold, which motivates us to analyze Nesterov\added{'s} acceleration from a Riemannian perspective.

In contrast to NAG, Nesterov's Accelerated Riemannian Gradient (NARG) method uses operations between tangent spaces and manifolds to ensure that the extrapolation lies on the manifold~\cite{ahn2020nesterov}, such as exponential operators, logarithmic operators, and parallel transport.
For most matrix manifolds, replacing the exponential operator with matrix factorization-based retraction enables an efficient implementation of NARG~\cite{duruisseaux2022variational}, for instance, sparse principal component analysis on the Stiefel manifold~\cite{huang2022extension}.
The intractable difficulty of NARG is that the inverse of the retraction usually does not have a closed-form expression, requiring an iterative algorithm to solve.
In particular, there is little work on the acceleration of the low-rank matrix manifold because the inverse of projection retraction may not be uniquely defined~\cite{absil2015low}.

Fortunately, although uncommon, orthographic retraction and its inverse admit simple and closed representations~\cite{absil2012projection,absil2015low}.
In this paper, we combine NAG and RGrad to develop a novel NARG method for low-rank matrix estimation.
To our knowledge, it is the first algorithm that uses orthographic retraction to establish subspace relations between iterative and extrapolated sequences.
The overall comparison is shown in Fig.~\ref{fig-all}.
Owing to momentum on the low-rank matrix manifold, NARG has the same computational complexity as RGrad~\cite{wei2016guarantees,wei2020guarantees}, with an advantage in convergence.
Based on the efficient implementation of orthographic retraction, the computational complexity of NARG is lower than that of NAG.

Our contributions can be summarized into three folds.
1) We firstly present a first-order perturbation analysis of the retractions, which provides a recursive representation of the iterative error.
2) By analyzing the relation of the spectral radius of the iterative matrix w.r.t. the parameters, we accurately estimate the linear convergence rate of all the algorithms in Fig.~\ref{fig-all}.
3) The convergence rate of NRAG+R, which uses the \emph{Adaptive Restart} Scheme, can match the theoretical optimal spectral radius.

\begin{figure}[htbp]
\centering
\subfigure[Illustration of Grad]{\label{fig:Grad-Algorithm}
\begin{minipage}[t]{0.47\linewidth}
\centering
\includegraphics[width=0.7\linewidth]{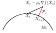}
\end{minipage}
}
\subfigure[Illustration of NAG]{\label{fig:NAG}
\begin{minipage}[t]{0.47\linewidth}
\centering
\includegraphics[width=0.7\linewidth]{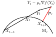}
\end{minipage}
}
\\
\subfigure[Illustration of RGrad]{\label{fig:RGrad}
\begin{minipage}[t]{0.47\linewidth}
\centering
\includegraphics[width=0.9\linewidth]{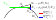}
\end{minipage}
}
\subfigure[Illustration of NARG]{\label{fig:NARG}
\begin{minipage}[t]{0.47\linewidth}
\centering
\includegraphics[width=0.9\linewidth]{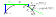}
\end{minipage}
}
\centering
\caption{The overall comparison: the traditional and accelerated methods are compared horizontally, and the Euclidean and Riemannian methods are compared vertically.}
\label{fig-all}
\end{figure}

For the convenience of readers, we compare representable algorithms for convergence and computational cost in Table~\ref{table:comparison}.
The IHT algorithm uses a constant stepsize, and the other\added{s} use the exact line search.
All algorithms exhibit local linear convergence when the condition (\ref{eq:Basin-of-Attraction}) holds.
In a nutshell, the results of NARG+R are dominant in both respects, which will be verified in subsequent experiments.

\begin{table*}[!h]
	\caption{Complexity comparisons between gradient algorithms and Nesterov's variants in the Euclidean and Riemannian settings.}
	\label{table:comparison}
	\centering
	\resizebox{\textwidth}{!}{
	\begin{tabular}{c|c|c|c} 
		\hline
		\textbf{Algorithm} & \textbf{Geometry} & \textbf{Local Linear Convergence Rate} & \textbf{\thead{Dominant per-iteration\\ computational complexity}} \\ [0.5ex] 
		\hline\hline
		IHT & Euclidean & $\max(1-\mu_t \lambda_{\min},\mu_t \lambda_{\max}-1)\geq\frac{\kappa-1}{\kappa+1}$ & $\mathcal O(n^3)$ \\ 
		\hline
		Grad & Euclidean & $\sqrt{1-\frac{\tilde\mu^2\lambda_{\max}\lambda_{\min}}{\tilde\mu(\lambda_{\max}+\lambda_{\min})-1}}\leq \frac{\kappa-1}{\kappa+1}$ & $\mathcal O(n^3)$ \\ 
		\hline
		NAG & Euclidean & $\sqrt{\eta_t(1-{\frac{4\lambda_{\min}}{\lambda_{\min}+3\lambda_{\max}}})}$ & $\mathcal O(n^3)$ \\ 
		\hline
		RGrad & Riemannian & $\sqrt{1-\frac{\tilde\mu^2\lambda_{\max}\lambda_{\min}}{\tilde\mu(\lambda_{\max}+\lambda_{\min})-1}}\leq \frac{\kappa-1}{\kappa+1}$ & $\mathcal O(n^2r)$ \\ 
		\hline
		NARG & Riemannian & $\sqrt{\eta_t(1-{\frac{4\lambda_{\min}}{\lambda_{\min}+3\lambda_{\max}}})}$ & $\mathcal O(n^2r)$ \\ 
		\hline
		NARG+R & Riemannian & $1-\sqrt{\frac{4\lambda_{\min}}{\lambda_{\min}+3\lambda_{\max}}}\approx 1-\sqrt{\frac{1}{\kappa}}$ & $\mathcal O(n^2r)$ \\ 
		\hline\hline
		\multicolumn{4}{l}{Parameters: $n=\min(n_1,n_2)$, Gradient stepsize $\mu_t$, Momentum parameter $\eta_t$, $\tilde\mu = \frac{\|\nabla_{\mathcal R} f(\boldsymbol X_t)\|_F^2}{\|\mathcal A(\nabla_{\mathcal R} f(\boldsymbol X_t))\|_2^2}$} \\
		\hline
	\end{tabular}}
\end{table*}

\subsection{Notation and Organization}

Throughout the paper, vectors are denoted by lowercase letters (e.g., $\boldsymbol{x}$), matrices by uppercase letters (e.g., $\boldsymbol{X}$), operators by calligraphic letters (e.g., $\mathcal P$), and set of matrices by double-stroke letters (e.g., $\mathbb {R}^{n_1\times n_2}$).
We utilize $\boldsymbol I_{n}$ as the $n$-by-$n$ identity matrix and abbreviate as $\boldsymbol I$ without size if the context is clear.
Let $\mathbb O^{p,r}=\{\boldsymbol U\in\mathbb R^{p\replaced{\times}{,} r}:\boldsymbol U^\top \boldsymbol U=\boldsymbol I_r\}$ represent a set of matrices with orthogonal columns.
For $\boldsymbol U\in\mathbb O^{p,r}$, $\boldsymbol U_\perp\in\mathbb O^{p,p-r}$ and $P_{\boldsymbol U}=\boldsymbol U\boldsymbol U^\top$ respectively denote its orthonormal complement and projection matrix.
\added{We use $P_{\boldsymbol U}^\perp := P_{\boldsymbol U_\perp} = \boldsymbol I - P_{\boldsymbol U}$ to represent the projection matrix onto perpendicular subspace.}
\added{Let $\mathbb M_r=\{\boldsymbol X\in\mathbb R^{n_1\times n_2}| \text{rank}(\boldsymbol X)=r\}$ be the set of matrices with fixed rank $r$, which is the smooth submanifold embedded in $\mathbb R^{n_1\times n_2}$.}
Given \deleted{the matrix $\boldsymbol X$ in the manifold of rank $r$, i.e.} $\boldsymbol X\in\mathbb M_r$, $\mathbb T_{\boldsymbol X}\mathbb M_r$ and $\mathbb T_{\boldsymbol X}^\perp\mathbb M_r$ stand for the tangent space and normal space at $\boldsymbol X$.
We denote $\mathcal P_\mathbb S$ as the projection operator to the set $\mathbb S$.
Let the full SVD of $n_1$-by-$n_2$ matrix $\boldsymbol Y$ with $\text{rank}(\boldsymbol Y) = n:=\min(n_1,n_2)$ be $\boldsymbol Y=\boldsymbol U_{\boldsymbol Y} \boldsymbol \Sigma_{\boldsymbol Y} \boldsymbol V_{ \boldsymbol Y}^\top$, where $\boldsymbol U_{\boldsymbol Y}\in\mathbb O^{n_1,n},\boldsymbol V_{\boldsymbol Y}\in\mathbb O^{n_2,n} $ and $\boldsymbol \Sigma_{\boldsymbol Y}=\text{diag}(\sigma_1,\ldots,\sigma_n)$ is diagonal matrix with descending order.
The projection of $\boldsymbol Y$ to $\mathbb M_r$ (a.k.a. truncated SVD) is defined as
\begin{eqnarray}\label{eq:TSVD}
\mathcal P_r(\boldsymbol Y):=[\boldsymbol U_{\boldsymbol Y}]_{:,1:r}[\boldsymbol \Sigma_{\boldsymbol Y}]_{1:r,1:r} [\boldsymbol V_{\boldsymbol Y}]_{:,1:r}^\top\in\mathbb M_r,
\end{eqnarray}
where $[\cdot]_{1:r,1:c}$ represents a submatrix composed of some rows and columns of subscript indexes, which is consistent with the expression of Matlab.
As some common matrix operations, $\|\cdot\|$, $\|\cdot\|_F$, $\text{vec}(\cdot)$ and $\otimes$ denote the spectral norm, Frobenius norm, vectorization and Kronecker product of the matrix, respectively. 
In the low-rank estimation problem, $\boldsymbol E_t=\boldsymbol X_t-\boldsymbol X_\star$ represents the residual between the estimate at the $t$-th iteration and the optimal solution.

The organization of this paper is as follows.
The local asymptotic convergence analysis of algorithm Grad with constant stepsize and the exact line search are presented in Sect.~\ref{sec:convergence-IHT-constant}.
Sect.~\ref{sec:convergence-NAG} discusses the convergence analysis of NAG.
In Sect.~\ref{sec:convergence-Riemannian-optimization}, we establish the Riemannian versions of Grad and NAG, coined RGrad and NARG, and adopt an adaptive restart scheme to improve the convergence rate of NARG heuristically in Sect.~\ref{sec:adaptive-restart}.
Sect.~\ref{sec:numerical-examples} illustrates the effectiveness of NARG by numerical studies.
Sect.~\ref{sec:conclusion} summarizes our work, followed by the proofs in the Appendix.

\section{Local convergence of Grad Algorithm}
\label{sec:convergence-IHT-constant}
For low-rank matrix estimation such as MS and MC, this section gives a unified representation of the gradient descent algorithm.
With the perturbation analysis of truncated SVD, we derive the convergence analysis by the spectral radius of the iterative matrix.

\subsection{Vectorization of the gradient}
Low-rank matrix estimation is usually a type of least-squares problem (\ref{eq:lowrank}) with the fixed-rank constraint, and its gradient can be written as
\begin{eqnarray*}
\nabla f(\boldsymbol X_t) = \mathcal A^*(\mathcal A(\boldsymbol X_t-\boldsymbol X_\star))=\mathcal A^*(\mathcal A(\boldsymbol E_t)),
\end{eqnarray*}
where $\mathcal A^*$ is the adjoint operator of the linear operator $\mathcal A$.
The vectorization of the gradient and residual matrix satisfies the following linear relationship
\begin{eqnarray}\label{eq:vec-grad}
\text{vec}(\nabla f(\boldsymbol X_t)) = \boldsymbol \Theta\text{vec}(\boldsymbol E_t) = \boldsymbol \Theta\boldsymbol e_t,
\end{eqnarray}
where $\boldsymbol \Theta$ is related to the specific estimation task.
Subsequently, we take MC and MS as examples to introduce how to construct the matrix $\boldsymbol \Theta$.

\textbf{MS}:
The linear measurement operator in MS $\mathcal A:\mathbb R^{n_1\times n_2} \mapsto \mathbb R^m$ is defined as follows:
\begin{eqnarray*}
\mathcal A(\boldsymbol X)=[\langle \boldsymbol A_i, \boldsymbol X\rangle]_{1\leq i \leq m}\in\mathbb R^m,
\end{eqnarray*}
where $\{\boldsymbol A_i\in\mathbb R^{n_1\times n_2}\}_{i=1}^m$ is the known matrix set, and its adjoint operator $\mathcal A^*:\mathbb R^m\mapsto \mathbb R^{n_1\times n_2}$ is defined as $\mathcal A^*(\boldsymbol y)=\sum_{i=1}^m \boldsymbol y_i \boldsymbol A_i$.
By vectorization, we have
\begin{eqnarray*}
\text{vec}(\mathcal A^*(\mathcal A(\boldsymbol E_t)))
=\sum_i [\text{vec}(\boldsymbol A_i)^\top\otimes\text{vec}(\boldsymbol A_i)]\text{vec}(\boldsymbol E_t),
\end{eqnarray*}
thus, the matrix $\boldsymbol \Theta$ in MS is expressed as follows:
\begin{eqnarray}\label{eq:Theta-MS}
\begin{aligned}
\boldsymbol \Theta_{\mathsf{MS}}
=\sum_i [\text{vec}(\boldsymbol A_i)^\top\otimes\text{vec}(\boldsymbol A_i)].
\end{aligned}
\end{eqnarray}

\textbf{MC}: 
The purpose of MC is to complete the entire low-rank matrix $\boldsymbol X_\star$ based on partial observations $\mathcal P_\Omega(\boldsymbol X_\star)$.
The corresponding loss function is $f(\boldsymbol X )=\|\mathcal P_\Omega(\boldsymbol X-\boldsymbol X_\star)\|_F^2$, where the projection $\mathcal P_\Omega$ to the observation index subset $\Omega$ is defined as
\begin{eqnarray*}
[\mathcal P_\Omega(\boldsymbol X)]_{i,j}=\begin{cases}
\boldsymbol X_{i,j}, &\text{if}~(i,j)\in\Omega,\\
0,&\text{otherwise}.
\end{cases}
\end{eqnarray*}

MC can be regarded as a variant of MS~\cite{chi2019nonconvex}, and the measurement matrices $\boldsymbol A_{i,j}\in\mathbb R^{n_1\times n_2}$ are set to
\begin{eqnarray*}
\boldsymbol A_{i,j}=\boldsymbol e_i^{n_1}\boldsymbol e_j^{n_2\top}=\begin{cases}
1, &\text{if}~(i,j)\in\Omega,\\
0,&\text{otherwise}.
\end{cases}
\end{eqnarray*}
\added{where $\boldsymbol e_i^{n_1}$ represents the $i$-th column of the identity matrix $\boldsymbol I_{n_1}$.}
Hence, the summation (\ref{eq:Theta-MS}) is equal to the following binary diagonal matrix, whose main diagonal elements are $\boldsymbol \omega=\text{vec}(\sum_{(i,j)\in\Omega} \boldsymbol A_{i,j})\in\mathbb R^{n_1n_2}$
\begin{eqnarray}\label{eq:Theta-MC}
\boldsymbol \Theta_{\mathsf{MC}}=\sum_{(i,j)\in\Omega} [\text{vec}(\boldsymbol A_{i,j})^\top\otimes\text{vec}(\boldsymbol A_{i,j})]=\text{diag}(\boldsymbol \omega).
\end{eqnarray}

We denote the cardinality by $|\Omega|$, i.e., the number of observed elements.
The selection matrix $\boldsymbol S_{\Omega}=(\boldsymbol e_{i_k}^{n_1n_2})\in\mathbb R^{n_1n_2\times |\Omega|}$ is constructed by selecting some columns from the identity matrix whose column indices satisfy $\boldsymbol\omega_{i_k}=1$.
On the contrary, let the complementary set $\bar\Omega$ be unobserved, then the matrix $\boldsymbol S_{\bar\Omega}$ consisting of the remaining columns satisfies
\begin{eqnarray}\label{eq:sample-matrix}
\begin{aligned}
&\boldsymbol S_{\Omega}^\top \boldsymbol S_{\Omega} =\boldsymbol I_{|\Omega|},
\boldsymbol S_{\bar\Omega}^\top \boldsymbol S_{\bar\Omega} =\boldsymbol I_{n_1n_2-|\Omega|},\\
&\boldsymbol S_{\Omega} \boldsymbol S_{\Omega}^\top+\boldsymbol S_{\bar\Omega} \boldsymbol S_{\bar\Omega}^\top=\boldsymbol I_{n_1n_2},\\
&\text{vec}(\mathcal P_{\Omega}(\boldsymbol X))=\boldsymbol S_\Omega \boldsymbol S_\Omega^\top\boldsymbol x=\boldsymbol \Theta_{\mathsf{MC}}\boldsymbol x,\\
&\text{vec}(\mathcal P_{\bar\Omega}(\boldsymbol X))=\boldsymbol S_{\bar\Omega} \boldsymbol S_{\bar\Omega}^\top\boldsymbol x=(\boldsymbol I_{n_1n_2}-\boldsymbol \Theta_{\mathsf{MC}})\boldsymbol x.
\end{aligned}
\end{eqnarray}
\added{where $\boldsymbol x = \text{vec}(\boldsymbol X)$.}

\subsection{Grad Algorithm with constant stepsize}
The Grad Algorithm (see Algorithm~\ref{Alg-IHT}), a.k.a IHT~\cite{jain2010guaranteed}, is a typical projected gradient method for solving (\ref{eq:lowrank}). 
It first performs vanilla gradient descent with a constant stepsize $\mu_t\equiv\mu$, then ensures the low-rank constraint by truncating SVD. 
The Grad Algorithm is illustrated in Fig.~\ref{fig:Grad-Algorithm}.

	\begin{algorithm}[H]
		\caption{Grad Algorithm with constant stepsize}
		\begin{algorithmic}\label{Alg-IHT}
			\REQUIRE observation data $\boldsymbol y_\mathsf{ob}$, rank $r$, maximum iteration $T$, constant stepsize $\mu$,
			\STATE Initialize: $\boldsymbol X_0=\mathcal A^*(\boldsymbol y_{\mathsf{ob}})$,
			\FOR{$t = 0,1,...,T-1$}
			\STATE $\boldsymbol X_{t+1}=\mathcal P_r(\boldsymbol X_{t}-\mu\nabla f(\boldsymbol X_t))$,
			\ENDFOR
			\ENSURE $\boldsymbol X_{T}$.
		\end{algorithmic}
	\end{algorithm}

Qualitative convergence analysis of this algorithm and its variants \replaced{have}{has} been extensively studied, see~\cite{chi2019nonconvex}.
However, the accurate estimate of the convergence rate has not been systematically studied.
A recent framework~\cite{vu2021asymptotic} describes the asymptotic linear convergence of projected gradient descent.
Inspired by this, we combine the gradient and subspace to construct \replaced{a}{an} recursive \replaced{equation}{eqnarray} of vectorized errors, \replaced{where the spectral norm of the iteration matrix}{whose spectral norm} can estimate the local linear convergence of Algorithm~\ref{Alg-IHT}.
It is worth mentioning that this result applies to widespread low-rank models such as MC and MS and can \added{be} further \replaced{extended}{extend} to manifold versions in Sect.~\ref{sec:convergence-Riemannian-optimization}.

Local convergence for constrained least squares~\cite{vu2021asymptotic} requires the Lipschitz-continuous differentiability of the projection operator $\mathcal P_r$.
As an essential tool, the following lemma allows a first-order approximate expansion of the well-known smooth constraint $\mathbb M_r$.
\begin{lemma}[Perturbation Analysis of Truncated SVD~\cite{chunikhina2014performance,vu2019accelerating}]\label{lem:TSVD-Perturbation}
Let \(\boldsymbol X=\boldsymbol U\boldsymbol \Sigma \boldsymbol V^\top\) be SVD of matrix \(\boldsymbol X\) with rank \(r\).
Assuming that the perturbation matrix \(\boldsymbol N\) satisfies \(\|\boldsymbol N\|_F< \sigma_{\replaced{r}{\min}} (\boldsymbol X)/2\), the first-order perturbation expansion of truncated SVD can be formulated as
\begin{eqnarray}\label{eq:TSVD-expansion}
\begin{aligned}
\mathcal P_r(\boldsymbol X+\boldsymbol N)&=\boldsymbol X+\boldsymbol N-P_{\boldsymbol U}^\perp \boldsymbol N P_{\boldsymbol V}^\perp+\mathcal O(\|\boldsymbol N\|_F^2).\\
\end{aligned}
\end{eqnarray}
\end{lemma}
According to the above condition, we roughly judge the region of convergence called the {\em Basin of Attraction}.
\begin{eqnarray}\label{eq:Basin-of-Attraction}
 \|\boldsymbol X_t-\boldsymbol X_\star\|_F\asymp\sigma_{\replaced{r}{\min}}(\boldsymbol X_\star).
\end{eqnarray}
Under Lemma~\ref{lem:TSVD-Perturbation}, once condition (\ref{eq:Basin-of-Attraction}) holds, the subsequent iterations converge linearly, which is stated as follows.
\begin{theorem}[Convergence for Grad with constant stepsize]\label{th:MatIHT}
Let $\lambda_{\max}$ and $\lambda_{\min}$ correspond to the largest and smallest non-zero eigenvalues of $(\boldsymbol I-P_{\boldsymbol V_\star}^\perp\otimes P_{\boldsymbol U_\star}^\perp)\boldsymbol \Theta$, respectively.
\added{The stepsize $\mu$ satisfies $\|\mathcal I-\mu \mathcal A^*\mathcal A\|\leq 1$.}
And set 
\begin{eqnarray}\label{eq:IHT-rho}
\rho=\max(1-\mu \lambda_{\min},\mu \lambda_{\max}-1)\added{\leq 1}.
\end{eqnarray}
When condition (\ref{eq:Basin-of-Attraction}) holds, Algorithm~\ref{Alg-IHT} satisfies
\begin{eqnarray}
\|\boldsymbol X_{t+1}-\boldsymbol X_\star\|_F \leq \rho\|\boldsymbol X_t-\boldsymbol X_\star\|_F.
\end{eqnarray}
\end{theorem}
See Appendix~\ref{app:Th:MatIHT} for proof.
When $\boldsymbol \Theta$ and $\boldsymbol X_\star$ are fixed, the stepsize $\mu$ affects the convergence rate.
As shown in (\ref{eq:IHT-rho}), we give two special stepsizes,
\begin{eqnarray}\label{eq:IHT-step}
\mu_{\dagger}:=2/(\lambda_{\max}+\lambda_{\min})~\text{and}~ \mu_{\ddagger}:=2/\lambda_{\max},
\end{eqnarray}
corresponding to the optimal convergence rate and upper bound, respectively. 
Obviously, $\mu_\dagger<\mu_{\ddagger}$.
Furthermore, the closer $\mu$ is to $\mu_\dagger$, the faster the convergence.
Conversely, $\rho>1$ holds when $\mu\geq\mu_{\ddagger}$, which means Algorithm~\ref{Alg-IHT} does not converge.

\added{
\begin{remark}\label{remark-Basin}
The premise of satisfying condition (\ref{eq:Basin-of-Attraction}) is the convergence guarantee with proper initialization, which has received extensive investigations; see ~\cite{cai2018exploiting,chen2018harnessing,chi2019nonconvex} and references therein.
On the one hand, some mild statistical assumptions and optimization properties can provide global convergence guarantees for first-order algorithms~\cite{chi2019nonconvex}, such as the number of samples, restricted isometry property, matrix incoherence, and regularity condition.
On the other hand, proper initialization can speed up the process of satisfying the condition (\ref{eq:Basin-of-Attraction}).
Subsequent works also establish guarantees that the iteration sequence satisfies local convergence conditions for spectral initialization~\cite{chen2021spectral} and random initialization~\cite{chen2019gradient}, respectively.
It is worth mentioning that although the local convergence radius is different under different model assumptions, they are all equal to $\sigma_{r}(\boldsymbol X_\star)$ up to a constant, which is the same form as the condition (\ref{eq:Basin-of-Attraction}).
\end{remark}
}

\begin{remark}[$\mu=1$ for MC]
In this case, the iteration can be simplified to
\begin{eqnarray*}
\boldsymbol X_{t+1}=\mathcal P_{\bar\Omega}(\mathcal P_r(\boldsymbol X_{t}))+\boldsymbol X_{\mathsf{ob}}, ~\text{where}~\boldsymbol X_{\mathsf{ob}} = \mathcal P_{\Omega}(\boldsymbol X_\star),
\end{eqnarray*}
And the convergence rate is $\rho=1-\lambda_{\min}$, which is consistent with \cite{vu2019accelerating,vu2021local}, due to $\lambda(\boldsymbol S_{\bar\Omega}^\top (P_{\boldsymbol V_\star}^\perp\otimes P_{\boldsymbol U_\star}^\perp)\boldsymbol S_{\bar\Omega})=\sigma^2(\boldsymbol S_{\bar\Omega} (\boldsymbol V_\perp\otimes \boldsymbol U_\perp))$.
\end{remark}
\begin{remark}[Optimal convergence rate]\label{remark-opt}
When $\mu=\mu_\dagger$, the convergence rate $\rho=1-\mu_\dagger\lambda_{\min}$ is theoretically optimal, i.e., 
\begin{eqnarray*}
\|\boldsymbol X_t-\boldsymbol X_\star\|_F\leq\left(\frac{\kappa-1}{\kappa+1}\right)^t\|\boldsymbol X_0-\boldsymbol X_\star\|_F.
\end{eqnarray*}
where $\kappa=\lambda_{\max }/\lambda_{\min}$ is the condition number \added{of matrix $(\boldsymbol I-P_{\boldsymbol V_\star}^\perp\otimes P_{\boldsymbol U_\star}^\perp)\boldsymbol \Theta$}.
\end{remark}
\begin{remark}[Relation to general optimization problems]\label{remark-general}
In fact, for a general $\alpha$-strongly convex and $\beta$-smooth function, the convergence rate of the gradient method with stepsize $\mu=\frac{2}{\alpha+\beta}$ is $\frac{\kappa_f- 1}{\kappa_f+1}$, where $\kappa_f=\replaced{\frac{\beta}{\alpha}}{\frac{\alpha}{\beta}}$ is the condition number of the loss function.
Similarly, due to the non-expansiveness of projection, it also holds for a class of closed convex-constrained optimization problems.
However, it does not consider the geometric properties of constraints.
In contrast, Theorem~\ref{lem:TSVD-Perturbation} takes full advantage of subspaces of the low-rank constraint.
Particularly, when not restricting low-rank constraint, i.e., $P_{\boldsymbol V_\star}^\perp\otimes P_{\boldsymbol U_\star}^\perp=\boldsymbol O_{n_1n_2}$, the convergence rate degenerates to $\rho=\frac{\kappa_f-1}{\kappa_f+1}$.
\end{remark}

However, the optimal stepsize $\mu_\dagger$ requires the singular matrix pair $(\boldsymbol U_\star, \boldsymbol V_\star)$ to be known in advance, which is not practical in application.
Some heuristic adaptive stepsize approaches, such as Normalized IHT (NIHT)~\cite{tanner2013normalized}, have shown effectiveness in theory and practice, motivating us to estimate the convergence rate of exact line search under the low-rank constraint.

\subsection{Grad Algorithm with exact line search}
\label{sec:convergence-Grad}

Lemma~\ref{lem:TSVD-Perturbation} asserts that the low-rank matrix constraint can be locally transformed into linear constraints in subspace form.
Once condition (\ref{eq:Basin-of-Attraction}) is satisfied, applying orthogonal projection to gradient also achieves the same linear convergence rate
\begin{eqnarray}\label{eq:proj-gradient}
\nabla_{\mathcal R} f(\boldsymbol X)=P_{\boldsymbol U_{\boldsymbol X}}\nabla f(\boldsymbol X)+\nabla f(\boldsymbol X)P_{\boldsymbol V_{\boldsymbol X}}-P_{\boldsymbol U_{\boldsymbol X}}\nabla f(\boldsymbol X)P_{\boldsymbol V_{\boldsymbol X}}.
\end{eqnarray}
Similar to Lemma~\ref{lem:TSVD-Perturbation}, by the orthogonal relationship of the projection, we get
\begin{eqnarray*}
\mathcal P_r(\boldsymbol X_t-\mu_t\nabla f(\boldsymbol X_t))=\mathcal P_r(\boldsymbol X_t-\mu_t\nabla_{\mathcal R} f(\boldsymbol X_t))+\mathcal O(\|\boldsymbol E_{t}\|_F^2).
\end{eqnarray*}
In fact, $\nabla_{\mathcal R} f(\boldsymbol X)$ is the Riemannian gradient, which will be analyzed in Sect.~\ref{sec:convergence-Riemannian-optimization}.
Replacing $\nabla f(\boldsymbol X)$ with $\nabla_{\mathcal R} f(\boldsymbol X)$, we get a first-order approximation
\begin{eqnarray*}
\begin{aligned}
\boldsymbol X_{t+1}&=\mathcal P_r(\boldsymbol X_{t}-\mu_t\nabla_\mathcal R f(\boldsymbol X_t))\\
&\stackrel{(a)}{=}\boldsymbol X_{t}-\mu_t\nabla_\mathcal R f(\boldsymbol X_t)-P_{\boldsymbol U_\star}^\perp \boldsymbol E_{t}P_{\boldsymbol V_\star}^\perp+\mathcal O(\|\boldsymbol E_{t}\|_F^2)\\
&\stackrel{(b)}{=}\boldsymbol X_{t}-\mu_t\nabla_\mathcal R f(\boldsymbol X_t)+\mathcal O(\|\boldsymbol E_{t}\|_F^2)\\
&\stackrel{(c)}{\approx}\boldsymbol X_{t}-\mu_t\nabla_\mathcal R f(\boldsymbol X_t),
\end{aligned}
\end{eqnarray*}
where $(a)$ is similar to Appendix~\ref{app:Th:MatIHT} and $(b)$ uses \deleted{the} Lemma~\ref{lem:subspace project-error}.
The process $(c)$ approximates the optimization problem with low-rank \replaced{constraint}{constrain} into an unconstrained quadratic problem based on subspaces.
The inner product property of the linear operator $\mathcal A$ can convert the loss function (\ref{eq:lowrank}) into vector form
\begin{eqnarray*}
\begin{aligned}
f(\boldsymbol X)&=\frac{1}{2}\|\mathcal A(\boldsymbol X)-\boldsymbol y\|_2^2
=\frac{1}{2}\|\mathcal A(\boldsymbol X-\boldsymbol X_\star)\|_2^2
=\frac{1}{2}\langle\mathcal A(\boldsymbol X-\boldsymbol X_\star),\mathcal A(\boldsymbol X-\boldsymbol X_\star)\rangle\\
&=\frac{1}{2}\langle \boldsymbol X-\boldsymbol X_\star,\mathcal A^*(\mathcal A(\boldsymbol X-\boldsymbol X_\star))\rangle
=\frac{1}{2}(\boldsymbol x-\boldsymbol x_\star)^\top \boldsymbol \Theta (\boldsymbol x-\boldsymbol x_\star).
\end{aligned}
\end{eqnarray*}

Exact line search aims to minimize $f(\boldsymbol X_t -\mu \nabla_\mathcal R f(\boldsymbol X_t))$ w.r.t. $\mu$.
By vectorization, the adaptive stepsize corresponds to the following problem
\begin{eqnarray*}
\mu_t =\operatorname*{argmin}_\mu \frac{1}{2}(\boldsymbol x_t -\mu \nabla_{\mathcal R} f(\boldsymbol x_t)-\boldsymbol x_\star)^\top\boldsymbol \Theta (\boldsymbol x_t -\mu \nabla_{\mathcal R} f(\boldsymbol x_t)-\boldsymbol x_\star).
\end{eqnarray*}
This problem is quadratic and convex w.r.t. $\mu$, and its explicit solution is easy to obtain by the properties of $\nabla_\mathcal R f(\boldsymbol x_t)$ and $\boldsymbol \Theta$ as
\begin{eqnarray}\label{eq:mu-close-form}
\mu_t = \frac{\nabla_\mathcal R f(\boldsymbol x_t)^\top \nabla_\mathcal R f(\boldsymbol x_t)}{\nabla_\mathcal R f(\boldsymbol x_t)^\top \boldsymbol \Theta \nabla_\mathcal R f(\boldsymbol x_t)}=\frac{\|\nabla_{\mathcal R} f(\boldsymbol X_t)\|_F^2}{\|\mathcal A(\nabla_{\mathcal R} f(\boldsymbol X_t))\|_2^2}.
\end{eqnarray}
The above stepsize is consistent with the Riemannian setting~\cite{vandereycken2013low}.
As mentioned previously, the local landscape of the low-rank matrix estimate is equivalent to a quadratic problem.
For the latter, the zigzag trajectory phenomenon is inseparable from the asymptotic property of the following lemma.
\begin{lemma}[Exact line search~\cite{gonzaga2016steepest,luenberger2021linear,huang2022asymptotic}]\label{lem:Asymptotic-Convergence-Quadratic}
For an unconstrained quadratic optimization problem
\begin{eqnarray*}
\min_{\boldsymbol x\in\mathbb R^n} f(\boldsymbol x)=\frac{1}{2} (\boldsymbol x-\boldsymbol x_\star)^\top \boldsymbol Q(\boldsymbol x-\boldsymbol x_\star),
\end{eqnarray*}
the gradient descent algorithm with exact line search is iterated as follows:
\begin{eqnarray*}
\boldsymbol x_{t+1}=\boldsymbol x_{t}-  \mu_t\nabla f(\boldsymbol x_t), 
~\text{where}~ \mu_t = \frac{\nabla f(\boldsymbol x_t)^\top \nabla f(\boldsymbol x_t)}{\nabla f(\boldsymbol x_t)^\top \boldsymbol Q \nabla f(\boldsymbol x_t)}.
\end{eqnarray*}
Then the stepsize sequence $\{\mu_t\}$ is oscillating and satisfies $(\mu_{2k-1},\mu_{2k})\to(\hat \mu,\check\mu)$ with asymptotic behaviour $\hat\mu^{-1}+\check\mu^{-1}=\lambda_{\max}(\boldsymbol Q)+\lambda_{\min}( \boldsymbol Q)$.
Moreover, the function value satisfies
\begin{eqnarray}\label{eq:Convergence-Quadratic}
f(\boldsymbol x_{t})\leq\left(\frac{\kappa_{\boldsymbol Q}-1}{\kappa_{\boldsymbol Q}+1}\right)^{2t}f(\boldsymbol x_0),
\end{eqnarray}
where $\kappa_{\boldsymbol Q}:=\lambda_{\max}(\boldsymbol Q)/\lambda_{\min}(\boldsymbol Q)$ is the condition number of the matrix $\boldsymbol Q$.
\end{lemma}
\begin{remark}[Convergence for the quadratic problem]\label{remark-Convergence-Quadratic}
The inequality (\ref{eq:Convergence-Quadratic}) can be proved by Kantorovich's inequality.
Let the eigenvectors corresponding to the largest and smallest eigenvalues of the matrix $\boldsymbol Q$ be $\boldsymbol v_1, \boldsymbol v_n$, respectively.
Then the equality in (\ref{eq:Convergence-Quadratic}) holds if and only the gradient $\nabla f(\boldsymbol x_0)=\boldsymbol Q(\boldsymbol x_0-\boldsymbol x_\star)$ at the initial point $\boldsymbol x_0$ can be expressed as a linear combination $k_1\boldsymbol v_1+k_n\boldsymbol v_n$ with $|k_1/k_n|=1$.
In this case, it is easy to get $\hat\mu=\check\mu=2/(\lambda_{\max}(\boldsymbol Q)+\lambda_{\min}(\boldsymbol Q))$, which indicates that the worst-case convergence of the gradient method with exact line search is equivalent to that based on the optimal constant stepsize.
Therefore, $(\frac{\kappa_{\boldsymbol Q}-1}{\kappa_{\boldsymbol Q}+1})^2$ is a rough upper bound for the judgment of the convergence rate.
To characterize the convergence rate finely, we substitute the asymptotic property of the stepsize into the spectral radius of $\boldsymbol I-\mu_t \boldsymbol Q$ to obtain
\begin{eqnarray}\label{eq:quadratic-theta}
\begin{aligned}
\bar\rho := \sqrt{\rho(\boldsymbol I-\hat\mu \boldsymbol Q)\rho(\boldsymbol I-\check\mu \boldsymbol Q)}
\approx \sqrt{1-\frac{\tilde\mu^2\lambda_{\max}(\boldsymbol Q)\lambda_{\min}(\boldsymbol Q)}{\tilde\mu(\lambda_{\max}(\boldsymbol Q)+\lambda_{\min}(\boldsymbol Q))-1}},
\end{aligned}
\end{eqnarray}
where $\tilde\mu = \frac{\nabla f(\boldsymbol x_0)^\top \nabla f(\boldsymbol x_0)}{\nabla f(\boldsymbol x_0)^\top \boldsymbol Q \nabla f(\boldsymbol x_0)}$ is related to the initial point $\boldsymbol x_0$.
\end{remark}

We design the following Algorithm~\ref{Alg-Grad} based on the exact line search.

	\begin{algorithm}[H]
		\caption{Grad Algorithm with exact line search}
		\begin{algorithmic}\label{Alg-Grad}
			\REQUIRE observation $\boldsymbol y_\mathsf{ob}$, rank $r$, maximum iteration $T$.
			\STATE Initialize: $\boldsymbol X_0=\mathcal A^*(\boldsymbol y_{\mathsf{ob}})$.
			\FOR{$t = 0,1,...,T-1$}
			\STATE compute $\nabla_{\mathcal R} f(\boldsymbol X_t)=\nabla f(\boldsymbol X_t)-P_{\boldsymbol U_t}^\perp\nabla f(\boldsymbol X_t)P_{\boldsymbol V_t}^\perp$,
			\STATE exact line search rule $\mu_t=\frac{\|\nabla_{\mathcal R} f(\boldsymbol X_t)\|_F^2}{\|\mathcal A(\nabla_{\mathcal R} f(\boldsymbol X_t))\|_2^2}$,
			\STATE $\boldsymbol X_{t+1}=\mathcal P_r(\boldsymbol X_{t}-\mu_t\nabla_{\mathcal R} f(\boldsymbol X_t))$,
			\ENDFOR
			\ENSURE $\boldsymbol X_{T}$.
		\end{algorithmic}
	\end{algorithm}
	
Similar to the asymptotic property of exact line search in Lemma~\ref{lem:Asymptotic-Convergence-Quadratic}, the following proposition uses the spectral radius to estimate the convergence rate of Algorithm~\ref{Alg-Grad} accurately.
\begin{proposition}[Convergence for Grad with exact line search]\label{prop:NIHT-Convergence}
Let $\lambda_{\max}$ and $\lambda_{\min}$ correspond to the largest and smallest non-zero eigenvalues of $(\boldsymbol I-P_{\boldsymbol V_\star}^\perp\otimes P_{\boldsymbol U_\star}^\perp)\boldsymbol \Theta$, respectively.
The condition number is denoted as $ \kappa:=\lambda_{\max}/\lambda_{\min}$.
When condition (\ref{eq:Basin-of-Attraction}) holds, the function value of Algorithm~\ref{Alg-Grad} satisfies
\begin{eqnarray*}
f(\boldsymbol X_{t+1})\leq\left(\frac{\kappa-1}{\kappa+1}\right)^{2}f(\boldsymbol X_t).
\end{eqnarray*}
Furthermore, the residual satisfies the following recursion
\begin{eqnarray*}
\|\boldsymbol X_{t+1}-\boldsymbol X_\star\|_F \leq \sqrt{1-\frac{\tilde\mu^2\lambda_{\max}\lambda_{\min}}{\tilde\mu(\lambda_{\max}+\lambda_{\min})-1}}\|\boldsymbol X_{t}-\boldsymbol X_\star\|_F,
\end{eqnarray*}
where $\tilde\mu = \frac{\|\nabla_{\mathcal R} f(\boldsymbol X_t)\|_F^2}{\|\mathcal A(\nabla_{\mathcal R} f(\boldsymbol X_t))\|_2^2}\in[\lambda_{\max}^{-1},\lambda_{\min}^{-1}]$.
\end{proposition}
See Appendix~\ref{app:prop:NIHT-Convergence} for proof.
By a simple algebraic inequality, we have
\begin{eqnarray*}
\sqrt{1-\frac{\tilde\mu^2\lambda_{\max}\lambda_{\min}}{\tilde\mu(\lambda_{\max}+\lambda_{\min})-1}}\leq\frac{\kappa-1}{\kappa+1},
\end{eqnarray*}
which means that the worst convergence rate of exact line search is precisely that of the optimal constant stepsize, see Remark\added{s}~\ref{remark-opt} and~\ref{remark-Convergence-Quadratic}.

\begin{remark}[Related Work]\label{remark-NIHT}
The stepsize we use differs from the classic NIHT~\cite{tanner2013normalized} in the projection direction.
NIHT uses the projection matrix composed of the first $r$ left and right singular vectors as the search restriction direction to improve the correction of singular values, as follows:
\begin{eqnarray*}
\begin{aligned}
\mu_t^u:=\frac{\|P_{\boldsymbol U_t}\nabla f(\boldsymbol X_t)\|_F^2}{\|\mathcal A(P_{\boldsymbol U_t}\nabla f(\boldsymbol X_t))\|_2^2},
\mu_t^v:=\frac{\|\nabla f(\boldsymbol X_t)P_{\boldsymbol V_t}\|_F^2}{\|\mathcal A(\nabla f(\boldsymbol X_t)P_{\boldsymbol V_t})\|_2^2},
\mu_t^{uv}:=\frac{\|P_{\boldsymbol U_t}\nabla f(\boldsymbol X_t)P_{\boldsymbol V_t}\|_F^2}{\|\mathcal A(P_{\boldsymbol U_t}\nabla f(\boldsymbol X_t)P_{\boldsymbol V_t})\|_2^2}.
\end{aligned}
\end{eqnarray*}
In contrast, the projection gradient (\ref{eq:proj-gradient}) takes into account all three directions to set the stepsize (\ref{eq:mu-close-form}).
\end{remark}

\section{Local convergence of NAG Algorithm}
\label{sec:convergence-NAG}

In this section, we will improve the local linear convergence rate of the Grad algorithm by introducing momentum.
As demonstrated in the following iterations, the core of NAG is to use the extrapolated trend generated by momentum to speed up the first-order optimization method
\begin{eqnarray}\label{eq:Euclidean-Nesterov}
\begin{aligned}
\boldsymbol Y_{t+1}&=\boldsymbol X_{t+1}+\eta_t(\boldsymbol X_{t+1}-\boldsymbol X_{t}),\\
\boldsymbol X_{t+1}&=\boldsymbol Y_{t+1}-\mu_t\nabla f(\boldsymbol Y_{t+1}).
\end{aligned}
\end{eqnarray}
This idea is widely used in various fields and has fascinating interpretations, such as variational framework~\cite{wibisono2016variational}, integral quadratic constraint~\cite{lessard2016analysis}.
Similar to~\cite{vu2019accelerating}, we employ this general acceleration technique for low-rank matrix estimation and obtain the following algorithm.
An illustration of the NAG is shown in Fig.~\ref{fig:NAG}.

	\begin{algorithm}[H]
		\caption{Nesterov's Accelerated Gradient (NAG)}
		\begin{algorithmic}\label{Alg-NAG}
			\REQUIRE observation $\boldsymbol y_\mathsf{ob}$, rank $r$, maximum iteration $T$.
			\STATE Initialize: $\boldsymbol X_{-1}=\boldsymbol X_{0}=\mathcal A^*(\boldsymbol y_{\mathsf{ob}})$.
			\FOR{$t = 0,1,...,T-1$}
			\STATE compute extrapolation: $\boldsymbol Y_{t}=\boldsymbol X_{t}+\eta_t(\boldsymbol X_{t}-\boldsymbol X_{t-1})$ with proper $\eta_t$,
			\STATE compute $\nabla_{\mathcal R} f(\boldsymbol Y_t)=\nabla f(\boldsymbol Y_t)-P_{\boldsymbol U_t}^\perp\nabla f(\boldsymbol Y_t)P_{\boldsymbol V_t}^\perp$,
			\STATE exact line search rule $\mu_t=\frac{\|\nabla_{\mathcal R} f(\boldsymbol Y_t)\|_F^2}{\|\mathcal A(\nabla_{\mathcal R} f(\boldsymbol Y_t))\|_2^2}$,
			\STATE $\boldsymbol X_{t+1}=\mathcal P_r(\boldsymbol Y_{t}-\mu_t\nabla_{\mathcal R} f(\boldsymbol Y_t))$,
			\ENDFOR
			\ENSURE $\boldsymbol X_{T}$.
		\end{algorithmic}
	\end{algorithm}

We analyze the relationship between the spectral radius of the iterative matrix and $(\mu_t,\eta_t)$ in detail and generalize the convergence analysis in~\cite{kim2018adaptive} to low-rank matrix estimation as the following theorem.
\begin{theorem}[Optimal Convergence Rate of NAG]\label{th:MatNAG}
Let $\lambda_{\max}$ and $\lambda_{\min}$ correspond to the largest and smallest non-zero eigenvalues of $(\boldsymbol I-P_{\boldsymbol V_\star}^\perp\otimes P_{\boldsymbol U_\star}^\perp)\boldsymbol \Theta$, respectively.
The parameter pair $(\mu_t,\eta_t)$ represents the stepsize and momentum parameter,
\added{where $\mu_t$ satisfies $\|\mathcal I-\mu_t \mathcal A^*\mathcal A\|\leq 1$ and $\eta_t\in[0,1]$.}
Set $\boldsymbol H(\mu_t)=(\boldsymbol I-P_{\boldsymbol V_\star}^\perp\otimes P_{\boldsymbol U_\star}^\perp)(\boldsymbol I-\mu_t \boldsymbol \Theta)$.
Then the vectorization of the residuals $\boldsymbol e_{t}=\boldsymbol x_{t}-\boldsymbol x_{\star}$ corresponding to the sequence $\{\boldsymbol X_{t}\}$ generated by Algorithm~\ref{Alg-NAG} satisfies
\begin{eqnarray}\label{eq:iter-T}
\begin{aligned}
\begin{pmatrix}
\boldsymbol e_{t+1}\\
\boldsymbol e_{t}
\end{pmatrix}=\underbrace {\begin{pmatrix}
(1+\eta_t)\boldsymbol H(\mu_t) &-\eta_t\boldsymbol H(\mu_t)\\
\boldsymbol I&\boldsymbol 0
\end{pmatrix}}_{\boldsymbol T(\mu_t,\eta_t)}
\begin{pmatrix}
\boldsymbol e_{t}\\
\boldsymbol e_{t-1}
\end{pmatrix}
\end{aligned}
\end{eqnarray}
When condition (\ref{eq:Basin-of-Attraction}) holds, Algorithm~\ref{Alg-NAG} satisfies the following recursion
\begin{eqnarray*}
(\|\boldsymbol e_{t+1}\|_2^2+\|\boldsymbol e_{t}\|_2^2) \leq  \rho(\boldsymbol T(\mu_t,\eta_t))^2(\|\boldsymbol e_{t}\|_2^2+\|\boldsymbol e_{t-1}\|_2^2).
\end{eqnarray*}
When $(\mu_t,\eta_t)\equiv(\mu_\flat,\eta_\flat):=(\frac{4}{\lambda_{\min}+3\lambda_{\max}},\frac{1-\sqrt{\mu_\flat\lambda_{\min}}}{1+\sqrt{\mu_\flat\lambda_{\min}}})$,
Algorithm~\ref{Alg-NAG} achieves the optimal convergence rate, i.e.,
\begin{eqnarray}\label{eq:NAG-opt-convergence}
\rho_{\mathsf{opt}}=\min_{\mu,\eta}\rho(T(\mu,\eta))=1-\sqrt{\frac{4\lambda_{\min}}{\lambda_{\min}+3\lambda_{\max}}}.
\end{eqnarray}
\end{theorem}
See Appendix~\ref{app:th:MatNAG} for proof.
The optimal convergence rate matches the lower bounds for first-order optimization algorithms (up to constant), which is consistent with NAG for the quadratic problems~\cite{kim2018adaptive}.

As we all know, the NAG is not a strict descent algorithm, and the momentum parameter can affect its acceleration performance.
The acceleration mechanism can be cast as a linear dynamical system~\cite{o2015adaptive}.
We also calculate the optimal momentum parameter $\eta^-(\mu_t)$ w.r.t. $\mu_t$ in Appendix~\ref{app:th:MatNAG} and partition according to the behavior of iterative oscillations.
\begin{itemize}
\item $\eta_t< \eta^-(\mu_t)$: low momentum region, overdamped,
\item $\eta_t= \eta^-(\mu_t)$: optimal momentum, critically damped,
\item $\eta_t> \eta^-(\mu_t)$: high momentum region, underdamped.
\end{itemize}
Since the optimal momentum is usually unknown, the parameter monotonically increases from $0$ to $1$, i.e., $\eta_t:0\nearrow 1$.
It inevitably leads to performance degradation caused by the underdamped iteration.
To avoid the high momentum, the adaptive restart scheme~\cite{o2015adaptive} properly resets the parameter when the underdamped occurs, which we will discuss in Sect.~\ref{sec:adaptive-restart}.
Another effective Lazy strategy~\cite{liang2022improving} sets $\eta_t=\frac{t-1}{t+d}$ and demonstrates that the larger the parameter $d$, the better the algorithm performance of NAG.
Since the exact line search outperforms the optimal constant stepsize, the following corollary gives an upper bound of the convergence rate of NAG under the Lazy strategy.
\begin{corollary}[Convergence Rate of NAG]\label{col:convergence-NAG}
When the momentum parameter $\eta_t=\frac{t-1}{t+d}\geq\eta_{\flat}$, the iteration error generated by Algorithm~\ref{Alg-NAG} satisfies
\begin{eqnarray*}
\|\boldsymbol e_{t+1}\|_2+\|\boldsymbol e_{t}\|_2\leq\sqrt{\eta_t(1-\mu_{\flat}\lambda_{\min})}(\|\boldsymbol e_{t}\|_2+\|\boldsymbol e_{t-1}\|_2).
\end{eqnarray*}
\end{corollary}
The convergence rate of NAG gradually becomes slower under the Lazy strategy, so it does not belong to linear convergence.
Once we have the minimum number of iterations $t_0$ that satisfies condition (\ref{eq:Basin-of-Attraction}) and the total number of iterations $t_n$, we can roughly estimate the average convergence rate.
\begin{eqnarray}\label{eq:NAG-average-speed}
\begin{aligned}
\bar\rho_{\text{NAG}}&=(\prod_{t=t_0}^{t_n}\eta_t)^{1/(2(t_n-t_0+1))}\sqrt{1-\mu_{\flat}\lambda_{\min}}\\
&=(\prod_{i=0}^{d}\frac{t_0+i-1}{t_n+i-1})^{1/(2(t_n-t_0+1))}\sqrt{1-\mu_{\flat}\lambda_{\min}}
\end{aligned}
\end{eqnarray}

Note that Algorithm~\ref{Alg-NAG} needs two SVDs in the update step $\mu_t$ and truncated SVD.
By changing the order, the following algorithm only needs one SVD, and the convergence is the same as Algorithm~\ref{Alg-NAG}.
\begin{eqnarray}\label{eq:NAG-one-SVD}
\begin{aligned}
\boldsymbol Y_{t}&=\mathcal P_r(\boldsymbol Z_{t}),\\
\boldsymbol X_{t+1}&=\boldsymbol Y_{t}-\mu_t\nabla_{\mathcal R} f(\boldsymbol Y_t),\\
\boldsymbol Z_{t+1}&=\boldsymbol X_{t+1}+\eta_t(\boldsymbol X_{t+1}-\boldsymbol X_{t}).
\end{aligned}
\end{eqnarray}
However, those Euclidean methods, such as Algorithm~\ref{Alg-Grad} and Algorithm~\ref{Alg-NAG}, still suffer from the high computational cost caused by SVD.
To overcome the burden, we use the Riemannian gradient descent algorithm to solve the problem \replaced{(\ref{eq:lowrank})}{~\ref{eq:lowrank}}.

\section{Extension to Riemannian Optimization}
\label{sec:convergence-Riemannian-optimization}

In this section, we use the tools of the Riemannian manifold, such as subspace projection and retraction, to reduce the computational cost of algorithms in Sect.~\ref{sec:convergence-IHT-constant} and Sect.~\ref{sec:convergence-NAG}.
Further, we combine Nesterov's ideas and low-rank manifold tools to design algorithms with advantages in both time and space.

\subsection{Preliminaries on the geometry of low-rank matrix manifold}

Assume SVD of $\boldsymbol X\in\mathbb M_r\subset\mathbb R^{n_1\times n_2}$ is $\boldsymbol X=\boldsymbol U_{\boldsymbol X}\boldsymbol \Sigma_{\boldsymbol X} \boldsymbol V_{\boldsymbol X}^\top$.
The tangent space $\mathbb T_{\boldsymbol X}\mathbb M_r$ can be constructed by the direct sum of the row and column subspaces of $\boldsymbol X$.
\begin{eqnarray}\label{eq:tangent-space}
\begin{aligned}
\mathbb T_{X}\mathbb M_r=\{&\boldsymbol U_{\boldsymbol X}\boldsymbol M\boldsymbol V_{\boldsymbol X}^\top+\boldsymbol U_p \boldsymbol V_{\boldsymbol X}^\top+\boldsymbol U_{\boldsymbol X}\boldsymbol V_p^\top: \boldsymbol U_p^\top \boldsymbol U_{\boldsymbol X}=\boldsymbol V_p^\top \boldsymbol V_{\boldsymbol X}=\boldsymbol 0_{r\times r}\}.\\
\end{aligned}
\end{eqnarray}
where $\boldsymbol M\in\mathbb R^{r\times r},\boldsymbol U_p\in\mathbb R^{n_1\times r},\boldsymbol V_p\in\mathbb R^{n_2\times r}$.
The projection of any point $\boldsymbol Z\in\mathbb R^{n_1\times n_2}$ to $\mathbb T_{X}\mathbb M_r$ is
\begin{eqnarray}\label{eq:tangent-space-projection}
\mathcal P_{\mathbb T_{\boldsymbol X}\mathbb M_r}(\boldsymbol Z)=P_{\boldsymbol U_{\boldsymbol X}} \boldsymbol Z+\boldsymbol ZP_{\boldsymbol V_{\boldsymbol X}} -P_{\boldsymbol U_{\boldsymbol X}}\boldsymbol ZP_{\boldsymbol V_{\boldsymbol X}}.
\end{eqnarray}

\textbf{Optimization on the manifold:}
For a given differentiable function $f(\boldsymbol X)$, the general step for solving the optimization problem $\min_{\boldsymbol X\in\mathbb M_r} f(\boldsymbol X)$ on the manifold $\mathbb M_r$ is as follows
\begin{eqnarray}\label{eq:manifold-opt}
\boldsymbol X_{t+1}=\mathcal R_{\boldsymbol X_t}(-\eta_t \text{grad}f(\boldsymbol X_t)),
\end{eqnarray}
where $\text{grad}f(\boldsymbol X_t)=\mathcal P_{\mathbb T_{\boldsymbol X_t}\mathbb M_r}\nabla f(\boldsymbol X_t)$ represents the Riemannian gradient, which is obtained by projecting the Euclidean gradient to the tangent space.
A critical step is to pull the result from the tangent space back to the manifold through the retraction, denoted as $\mathcal R_{\boldsymbol X}(\cdot):\mathbb T_{\boldsymbol X}\mathbb M_r\to\mathbb M_r$.
We will briefly describe two common retractions: Projective retraction and Orthographic retraction.
In addition, we will introduce the inverse retraction, denoted $\mathsf{inv}\mathcal R_{\boldsymbol X}(\cdot):\mathbb M_r\to \mathbb T_{\boldsymbol X}\mathbb M_r$.
These concepts are visually described in Fig.~\ref{fig:retraction}.

\begin{figure}
\centering
\includegraphics[width=0.8\textwidth]{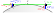}
\caption{Geometric comparison between projective retraction and orthographic retraction.}
\label{fig:retraction}
\end{figure}

\textbf{Projective retraction}:
For any tangent vector $\boldsymbol \delta\in \mathbb T_{X}\mathbb M_r$, the approximation problem corresponding to projective retraction can be solved by truncation SVD, i.e., $\mathcal R_{\boldsymbol X}^{\mathsf{proj}}(\boldsymbol \delta)=\mathcal P_r(\boldsymbol X+\boldsymbol \delta)$.
The computational cost on truncation SVD is $\mathcal O(n^3)$, where $n=\min{(n_1,n_2)}$.
Considering the representation of the tangent vector in (\ref{eq:tangent-space}), we can rewrite the matrix $\boldsymbol X+\boldsymbol \delta$ as a block matrix
\begin{eqnarray}\label{eq:block-matrix}
\boldsymbol X+\boldsymbol \delta=\begin{bmatrix}
\boldsymbol U_{\boldsymbol X} & \boldsymbol U_p
\end{bmatrix}
\begin{bmatrix}
\boldsymbol \Sigma_{\boldsymbol X} + \boldsymbol M& \boldsymbol I_r\\
\boldsymbol I_r & \boldsymbol {0}
\end{bmatrix}
\begin{bmatrix}
\boldsymbol V_{\boldsymbol X}&\boldsymbol V_p
\end{bmatrix}^\top.
\end{eqnarray}
Obviously, $\text{rank}(\boldsymbol X+\boldsymbol \delta)\leq 2r$.
The original SVD of (\ref{eq:block-matrix}) can be equivalently converted into two QR factorizations, one $2r\times 2r$ SVD and a few matrix multiplications~\cite{vandereycken2013low,boumal2022intromanifolds,wang2021fast}, with a total computational cost of $\mathcal O(n^2r + n^2+ nr^2+r^3)$, which greatly reduces the complexity when $r\ll n$.
The inverse projective retraction satisfies the following form.
\begin{eqnarray*}
\mathsf{inv}\mathcal R_{\boldsymbol X}^{\mathsf{proj}}(\boldsymbol X+\boldsymbol \delta)=(\boldsymbol X+\mathbb T_{\boldsymbol X}\mathbb M_r)\cap(\boldsymbol X+\boldsymbol \delta+\mathbb T_{\boldsymbol X+\boldsymbol \delta}^\perp\mathbb M_r)-\boldsymbol X.
\end{eqnarray*}
However, the inverse relies on tangent and normal spaces and does not have a closed-form representation.

\textbf{Orthographic retraction}:
For any tangent vector $\boldsymbol \delta\in \mathbb T_{\boldsymbol X}\mathbb M_r$, the orthographic retraction is defined as the closest point to $\boldsymbol X+\boldsymbol \delta$ in the vertical direction of the tangent space, i.e., $\boldsymbol X+\boldsymbol \delta+\mathbb T_{\boldsymbol X}^\perp \mathbb M_r\cap \mathbb M_r$, which has an explicit solution~\cite{zhang2018robust}
\begin{eqnarray}\label{eq:orth-retraction}
\mathcal R_{\boldsymbol X}^{\mathsf{orth}}(\boldsymbol \delta)=(\boldsymbol X+\boldsymbol \delta)\boldsymbol V_{\boldsymbol X}[\boldsymbol U_{\boldsymbol X}^\top (\boldsymbol X+\boldsymbol \delta)\boldsymbol V_{\boldsymbol X}]^{-1}\boldsymbol U_{\boldsymbol X}^\top (\boldsymbol X+\boldsymbol \delta).
\end{eqnarray}
The above formula only involves multiple matrix multiplications and a matrix inverse of $r\times r$, so it is efficient.
We are more concerned about constructing the tangent space of the next iteration from the current tangent space through SVD.
As suggested in~\cite{absil2015low}, the original SVD of (\ref{eq:orth-retraction}) can also be equivalently converted into two QR factorizations and a $r\times r$ SVD.
Specifically, we calculate $(\boldsymbol X+\boldsymbol \delta)\boldsymbol V_{\boldsymbol X} = \boldsymbol Q_1 \boldsymbol R_1, (\boldsymbol X+\boldsymbol \delta)^\top \boldsymbol U_{\boldsymbol X} = \boldsymbol Q_2 \boldsymbol R_2 $ and $\boldsymbol R_1[\boldsymbol U_{\boldsymbol X}^\top (\boldsymbol X+\boldsymbol \delta)\boldsymbol V_{\boldsymbol X}]^{-1}\boldsymbol R_2^\top={\boldsymbol U}_\flat\boldsymbol \Sigma_\flat{\boldsymbol V}_\flat^\top$.
Then orthographic retraction can be expressed as follows with SVD form
\begin{eqnarray}\label{eq:SVD-orth-retraction}
\mathcal R_{\boldsymbol X}^{\mathsf{orth}}(\boldsymbol \delta)=\underbrace{(\boldsymbol Q_1 \boldsymbol U_\flat)}_{\boldsymbol U_+}\underbrace{\boldsymbol \Sigma_\flat}_{\boldsymbol \Sigma_+} {\underbrace{(\boldsymbol Q_2 {\boldsymbol V}_\flat)}_{\boldsymbol V_+}}^\top.
\end{eqnarray}
The computational cost involved is consistent with the projective retraction.
Due to the orthogonal relationship to the tangent space, the inverse orthographic retraction is a simple projection to tangent space, i.e., $\mathsf{inv}\mathcal R_{\boldsymbol X}^{\mathsf{orth}}(\boldsymbol Y)=\mathcal P_{\mathbb T_{\boldsymbol X}\mathbb M_r}(\boldsymbol Y-\boldsymbol X)$.

\subsection{RGrad Algorithm with exact line search}
\label{subsec:convergence-RGrad}

Referring to~\cite{wei2016guarantees,wei2020guarantees,zhang2018robust}, we reformulate the RGrad under exact line search, see Algorithm~\ref{Alg-RGrad}.
As shown in Fig.~\ref{fig:RGrad}, the retraction $\mathcal R_{\boldsymbol X}$ can be choosen one of the projected $\mathcal R_{\boldsymbol X}^{\mathsf{proj}}$ and the orthogonal $\mathcal R_{\boldsymbol X}^{\mathsf{orth}}$.

	\begin{algorithm}[H]
		\caption{Riemannian gradient descent (RGrad)}
		\begin{algorithmic}\label{Alg-RGrad}
			\REQUIRE observation $\boldsymbol y_\mathsf{ob}$, rank $r$, maximum iteration $T$.
			\STATE Initialize: $\boldsymbol X_0=\mathcal A^*(\boldsymbol y_{\mathsf{ob}})$
			\FOR{$t = 0,1,...,T-1$}
			\STATE exact line search rule $\mu_t=\frac{\|\text{grad}~f(\boldsymbol X_t)\|_F^2}{\|\mathcal A(\text{grad}~f(\boldsymbol X_t))\|_2^2}$
			\STATE $\boldsymbol X_{t+1}=\mathcal R_{\boldsymbol X_t}(-\mu_t\text{grad}~f(\boldsymbol X_t))$
			\ENDFOR
			\ENSURE $\boldsymbol X_{T}$.
		\end{algorithmic}
	\end{algorithm}

Thanks to the efficient implementation of retraction, the computational complexity of RGard is lower than that of Grad.
In the following \replaced{lemma}{Lemma}, we introduce the perturbation analysis of retractions, which can reflect retractions and truncated SVD have the same first-order perturbation expansion.
\begin{lemma}[Perturbation analysis of retractions]\label{lem:retraction-Perturbation}
Let SVD of the matrix \(\boldsymbol X\in\mathbb M_r\) be \(\boldsymbol X=\boldsymbol U_{\boldsymbol X}\boldsymbol \Sigma_{\boldsymbol X} \boldsymbol V_{\boldsymbol X}^\top\).
Assuming that the perturbation matrix \(\boldsymbol N\) satisfies \(\|\boldsymbol N\|_F< \sigma_{\replaced{r}{\min}}(\boldsymbol X)/2\), the first-order perturbation expansion of retractions satisfies
\begin{eqnarray}\label{eq:retraction-expansion}
\begin{aligned}
\mathcal R_{\boldsymbol X}^{\mathsf{proj}}(\boldsymbol N)=&\mathcal P_{\mathbb T_{\boldsymbol X}}(\boldsymbol X+\boldsymbol N)+\mathcal O(\|\boldsymbol N\|_F^2),\\
\mathcal R_{\boldsymbol X}^{\mathsf{orth}}(\boldsymbol N)=&\mathcal P_{\mathbb T_{\boldsymbol X}}(\boldsymbol X+\boldsymbol N)+\mathcal O(\|\boldsymbol N\|_F^2).
\end{aligned}
\end{eqnarray}
\end{lemma}
See Appendix~\ref{app:lem:retraction-Perturbation} for proof.
The iterative matrix derived by RGrad is consistent with Grad in Appendix~\ref{app:col:MatRGrad}, which means the same convergence rate.

\subsection{Nesterov's Accelerated Riemannian Gradient}
\label{subsec:convergence-NARG}

From a manifold point of view, the extrapolation along the geodesic involves the exponential and the logarithmic maps~\cite{zhang2018towards,kim2022nesterov}.
Fortunately, these maps on the low-rank matrix manifold can often be replaced by first-order approximation, the retraction~\cite{vandereycken2013low,duruisseaux2022variational}.
Compared with projective retraction, orthographic retraction has an explicit inverse, i.e., projection $\mathcal P_{\mathbb T_{\boldsymbol X}\mathbb M_r}$.
Next, we establish the NARG Algorithm based on orthographic retraction.

	\begin{algorithm}[H]
		\caption{Nesterov's Accelerated Riemannian Gradient (NARG)}
		\begin{algorithmic}\label{Alg-NARG}
			\REQUIRE observation $\boldsymbol y_\mathsf{ob}$, rank $r$, maximum iteration $T$.
			\STATE Initialize: $\boldsymbol X_{-1}=\boldsymbol X_{0}=\mathcal A^*(\boldsymbol y_{\mathsf{ob}})$.
			\FOR{$t = 0,1,...,T-1$}
			\STATE choose proper $\eta_t$ and set extrapolation: $\boldsymbol Y_t=\mathcal R^{\mathsf {orth}}_{\boldsymbol X_t}(-\eta_t \mathsf{inv} \mathcal R^{\mathsf {orth}}_{\boldsymbol X_t}(\boldsymbol X_{t-1}))$,
			\STATE exact line search rule $\mu_t=\frac{\|\text{grad}~f(\boldsymbol Y_t)\|_F^2}{\|\mathcal A(\text{grad}~f(\boldsymbol Y_t))\|_2^2}$,
			\STATE $\boldsymbol X_{t+1}=\mathcal R^{\mathsf {orth}}_{\boldsymbol Y_t}(-\mu_t\text{grad}~f(\boldsymbol Y_t))$,
			\ENDFOR
			\ENSURE $\boldsymbol X_{T}$.
		\end{algorithmic}
	\end{algorithm}

The extrapolation sequence $\{\boldsymbol Y_t\}$ is strictly restricted to geodesics.
Through the inverse orthographic retraction, the following relationship reflects the linear extrapolation on the tangent space (\ref{eq:Euclidean-Nesterov}), i.e.,
\begin{eqnarray*}
\mathsf{inv} \mathcal R^{\mathsf {orth}}_{\boldsymbol X_t}(\boldsymbol Y_{t})=\boldsymbol X_{t}+\eta_t(\boldsymbol X_{t}-\mathsf{inv} \mathcal R^{\mathsf {orth}}_{\boldsymbol X_t}(\boldsymbol X_{t-1})),
\end{eqnarray*}
A visualization of the extrapolation process is presented in Fig.~\ref{fig:NARG}.
Compared to the non-accelerated Algorithm~\ref{Alg-RGrad}, Algorithm~\ref{Alg-NARG} seems to increase the computational cost caused by the extra operators, but the convergence rate is greatly improved.
Because the whole process alternates only between the manifold and tangent space, we can use (\ref{eq:SVD-orth-retraction}) to achieve a fast transfer of tangent space between $\{\boldsymbol X_t\}$ and $\{\boldsymbol Y_t\}$.
Regarding convergence, the iterative matrix of NAG is entirely consistent with NARG, and its derivation is shown in Appendix~\ref{app:col:MatNARG}.
Combined with Theorem~\ref{th:MatNAG}, NAG yields the same local linear convergence rate on Euclidean and manifold.
Essentially, this boils down to the same first-order expansion of iterations w.r.t. small perturbations.
\begin{remark}\label{remark-NARG}
Compared with the existing Euclidean acceleration, such as~\cite{vu2019accelerating,kyrillidis2014matrix}, we firstly establish\deleted{es} a bidirectional connection between tangent space and manifold by an orthographic retraction from the perspective of Riemannian geometry.
Following the framework~\cite{vu2021asymptotic}, a new perturbation analysis in Lemma~\ref{lem:retraction-Perturbation} can bridge the gap between Euclidean and manifold.
\end{remark}

\section{Adaptive Restart Scheme}
\label{sec:adaptive-restart}

In practice, optimal estimation of momentum parameters is often challenging.
The adaptive restart scheme~\cite{o2015adaptive} judges whether the momentum leads in the wrong direction through function or gradient conditions.
The convergence rate consistent with the optimal spectral radius is affirmed in applications such as strongly convex quadratic optimization~\cite{kim2018adaptive}, linear elliptic problem~\cite{park2021accelerated} and matrix completion~\cite{vu2019accelerating}.
We propose an adaptive restart scheme-based NARG algorithm with low complexity for low-rank matrix estimation; see Algorithm~\ref{Alg-NARG+R}.

	\begin{algorithm}[H]
		\caption{NARG with Adaptive Restart Scheme (NAGR+R)}
		\begin{algorithmic}\label{Alg-NARG+R}
			\REQUIRE observation $\boldsymbol y_\mathsf{ob}$, rank $r$, maximum iteration $T$.
			\STATE Initialize: $\boldsymbol X_{-1}=\boldsymbol X_{0}=\mathcal A^*(\boldsymbol y_{\mathsf{ob}})$.
			\FOR{$t = 0,1,...,T-1$}
			\IF{$\langle \nabla f(\boldsymbol Y_{t-1}), \boldsymbol X_t-\boldsymbol X_{t-1}\rangle >0$}
				\STATE$\tau=1$,
			\ELSE
				\STATE $\tau=\tau+1$,
			\ENDIF
			\STATE set extrapolation $\boldsymbol Y_t=\mathcal R^{\mathsf {orth}}_{\boldsymbol X_t}(-\eta_t \mathsf{inv} \mathcal R^{\mathsf {orth}}_{\boldsymbol X_t}(\boldsymbol X_{t-1}))$ with $\eta_t=\frac{\tau-1}{\tau+2}$,
			\STATE exact line search rule $\mu_t=\frac{\|\text{grad}~f(\boldsymbol Y_t)\|_F^2}{\|\mathcal A(\text{grad}~f(\boldsymbol Y_t))\|_2^2}$,
			\STATE $\boldsymbol X_{t+1}=\mathcal R^{\mathsf {orth}}_{\boldsymbol Y_t}(-\mu_t\text{grad}~f(\boldsymbol Y_t))$,
			\ENDFOR
			\ENSURE $\boldsymbol X_{T}$.
		\end{algorithmic}
	\end{algorithm}

Moreover, we establish the equivalence relation of gradient condition between Euclidean and manifold by the tangent space.
According to Lemma~\ref{lem:retraction-Perturbation}, we find that the sign of the gradient condition is consistent with its expansion in the tangent space, i.e.
\begin{eqnarray}\label{eq:restart-condition}
\begin{aligned}
&\text{sign}(\langle \text{grad}~f(\boldsymbol Y_{t-1}), \mathsf{inv} \mathcal R^{\mathsf {orth}}_{\boldsymbol Y_{t-1}}(\boldsymbol X_{t})-\mathsf{inv} \mathcal R^{\mathsf {orth}}_{\boldsymbol Y_{t-1}}(\boldsymbol X_{t-1})\rangle)\\
&=\text{sign}(\langle \nabla f(\boldsymbol Y_{t-1}), \boldsymbol X_t-\boldsymbol X_{t-1}\rangle).\\
\end{aligned}
\end{eqnarray}
See Appendix~\ref{app:col:Restart} for proof and Fig.~\ref{fig:Restart} for geometric interpretation.
Nevertheless, the projection introduces an additional computational cost, so we still use the traditional gradient condition in Algorithm~\ref{Alg-NARG+R}.
Besides, a practical trick is to perform a restart once condition (\ref{eq:Basin-of-Attraction}) holds, which can facilitate local approximation and iterative analysis of constraints.
\begin{figure}[!h]
\centering
\includegraphics[width=0.5\textwidth]{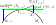}
\caption{Illustration of adaptive restart scheme from a manifold point of view.}\label{fig:Restart}
\end{figure}

The following corollary can be established with the optimal spectral radius in Theorem~\ref{th:MatNAG}.
\begin{corollary}[Adaptive Restart Scheme]\label{col:Restart}
Let $\lambda_{\max}$ and $\lambda_{\min}$ correspond to the largest and smallest non-zero eigenvalues of $(\boldsymbol I-P_{\boldsymbol V_\star}^\perp\otimes P_{\boldsymbol U_\star}^\perp)\boldsymbol \Theta$, respectively.
When condition (\ref{eq:Basin-of-Attraction}) holds, Algorithm~\ref{Alg-NARG+R} satisfies
\begin{eqnarray*}
\|\boldsymbol X_{t+1}-\boldsymbol X_\star\|_F \leq 1-\sqrt{\frac{4\lambda_{\min}}{\lambda_{\min}+3\lambda_{\max}}}\|\boldsymbol X_t-\boldsymbol X_\star\|_F.
\end{eqnarray*}
\end{corollary}

\section{Numerical Examples}
\label{sec:numerical-examples}

In this section, we provide numerical experiments to confirm our theoretical results, with the Matlab codes available at
\begin{center}
\url{https://github.com/pxxyyz/FastGradient}.
\end{center} 

\subsection{Convergence for Quadratic Problem}
We take the 2-D quadratic problem as an example.
Set the symmetric positive definite matrix $\boldsymbol Q=\begin{pmatrix} 10&1\\1&1\end{pmatrix}$ and the optimal point $\boldsymbol x_\star=(0,0)^\top$.
Let $\boldsymbol Q=\boldsymbol V\boldsymbol \Lambda \boldsymbol V^{-1}$ be the eigen decomposition.
Denote the rotation angle of the orthogonal basis as $\alpha=\min_{i=1,2}\arctan(\boldsymbol V_{i,2}/\boldsymbol V_{i,1})$, then $\boldsymbol x_0 = \boldsymbol Q^{-1}(\cos(\alpha+\theta), \sin(\alpha+\theta))^\top$ with $\theta\in[0,2\pi]$.
\begin{figure}[!h]
\centering
\includegraphics[width=0.8\textwidth]{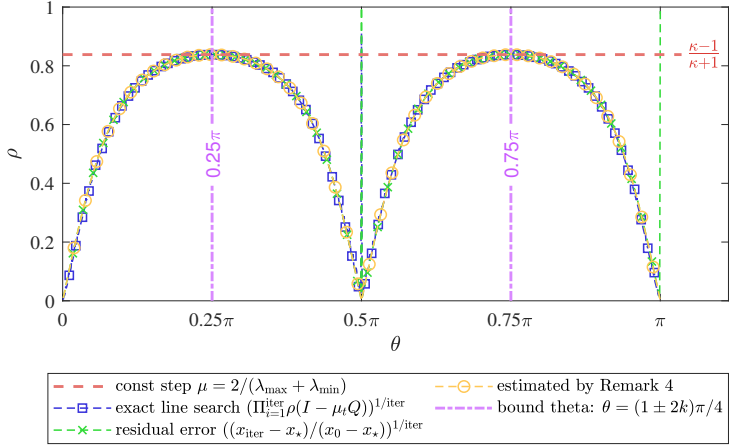}
\caption{Relationship between Convergence Rate and $\theta$ under Exact Line Search.}\label{fig:quadratic-theta}
\end{figure}

We know gradient $\nabla f(\boldsymbol x_0)=(\cos(\alpha+\theta), \sin(\alpha+\theta))^\top$ and stepsize $\tilde\mu=\langle \nabla f(\boldsymbol x_0), \boldsymbol Q\nabla f(\boldsymbol x_0)\rangle ^{-1}$.
The relationship (\ref{eq:quadratic-theta}) is verified in Fig.~\ref{fig:quadratic-theta}.
In particular, when $\theta = (1\pm 2k)\pi/4$, the convergence rate reaches the upper bound and is the same as the optimal constant stepsize, which verifies $\hat\mu=\check\mu=2/(\lambda_{\max}+\lambda_{\min})$ in Remark~\ref{remark-Convergence-Quadratic}.
When $\theta=k/2\pi$, $\boldsymbol x_0$ is located in the direction of the eigenvector, and the gradient points to $x_\star$.
The stepsize equals the corresponding eigenvalue, so it takes one step to reach $\boldsymbol x_\star$, consistent with~\cite{gonzaga2016steepest}.

\subsection{MC and MS}

Under different ranks and numbers of observations, we test the performance of algorithms and verify the validity of converged estimates.
Comparison methods include IHT, NIHT~\cite{tanner2013normalized}, Grad, RGrad, NAG, NARG, and NARG+R, where NIHT uses the first stepsize in Remark~\ref{remark-NIHT}, and the last uses the restart scheme.
The accelerated algorithms, including NAG and NARG, use the Lazy strategy with $d=2$.
Since there are two retractions, we use RGrad-Proj and RGrad-Orth to distinguish them.
The simulations of MC and MS are shown in Fig.~\ref{fig:MC-All} and Fig.~\ref{fig:MS-All}, respectively.
Here, the dashed line and the dash-dotted line compare the actual and estimated errors of the algorithm.
The brackets in the legend record the algorithm running time.
And $\rho$ is an estimate of the convergence rate, which reflects the slope of the linear decline.
It is worth mentioning that we use the optimal spectral radius of Theorem~\ref{th:MatNAG} as the optimal convergence rate of the accelerated algorithm.

\begin{figure*}[!t]
\centering
\includegraphics[width=1.0\textwidth]{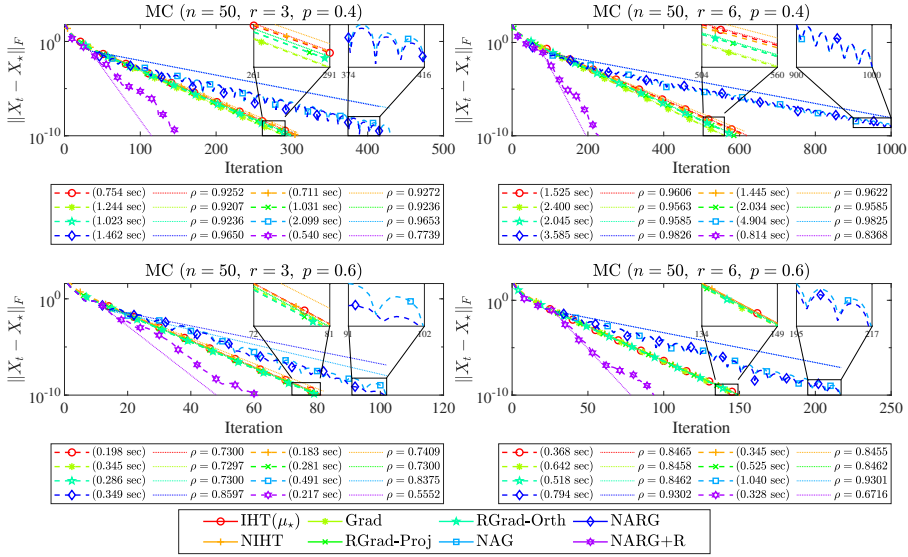}
\caption{(Log-scale) plot of error in MC.}\label{fig:MC-All}
\end{figure*}
\begin{figure*}[!t]
\centering
\includegraphics[width=1.0\textwidth]{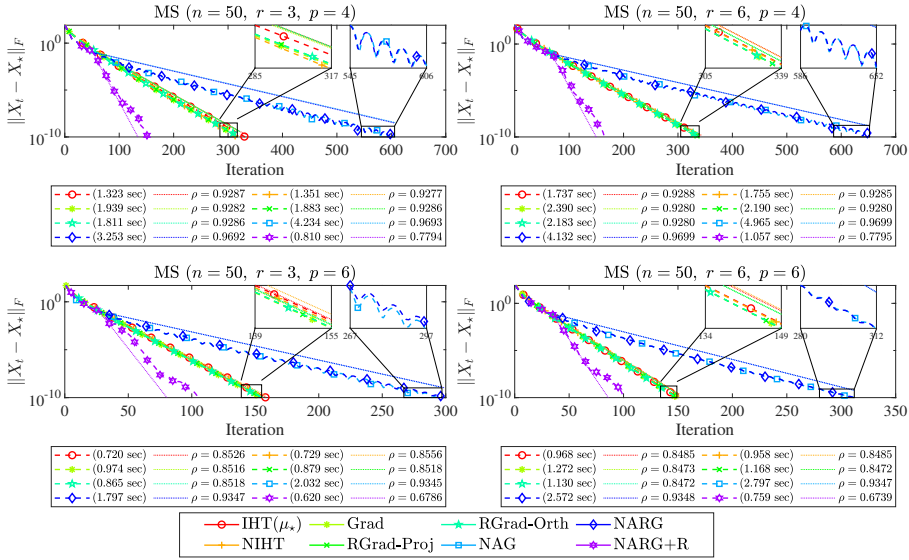}
\caption{(Log-scale) plot of error in MS.}\label{fig:MS-All}
\end{figure*}

From the comparison results, the estimation of the convergence rate of all algorithms is convincing.
The spectral radii of the first five methods are close, but the manifold-based methods dominate in terms of running time.
This also holds for comparing accelerated algorithms, namely NAG and NARG.
On the one hand, efficient retraction reduces the overall computational complexity and thus significantly reduces the running time under the premise of the same convergence rate.
On the other hand, the convergence rate is related to the iterative method but not the geometry, which also reflects Lemma~\ref{lem:retraction-Perturbation}.
Furthermore, for the accelerated methods, the high momentum under-damping causes the rippling behavior as the iterations increase.
The adaptive restart scheme can avoid ripples and matches the optimal convergence rate, confirming Corollary~\ref{col:Restart}.
Overall, the latest algorithm prevails on both sides.

\subsection{Oscillation caused by momentum}

To further illustrate the relationship between momentum and acceleration, we show the oscillatory effect under different parameters divided into two cases.
When the optimal parameter is known, as suggested in~\cite{o2015adaptive}, we set $\eta_t \equiv \frac{1-\sqrt{q}}{1+\sqrt{q}}$ and $q^\star=\kappa^{-1}$.
We observe the damping effect at different $q$ in Fig.~\ref{fig:Rippling-behavior}.
When $q<q^\star$, momentum higher than the optimal parameter will cause oscillations to slow the convergence rate.
Conversely, when $q>q^\star$, as $q$ increases, the gradually decreasing spectral radius makes the convergence slower.
Until $q=1$, i.e., $\eta_t\equiv0$, the momentum does not work, and the accelerated algorithm degenerates to the original algorithm.
The result of $q=q^\star$ is optimal and consistent with the trend of NARG+R.

\begin{figure}[!h]
\centering
\includegraphics[width=0.75\textwidth]{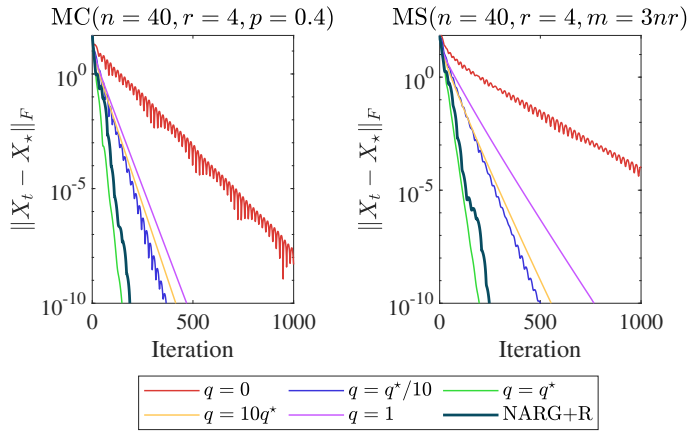}
\caption{Rippling behavior under different momentum parameter settings.}\label{fig:Rippling-behavior}
\end{figure}

When the optimal parameter is unknown, we set $\eta_t=\frac{t-1}{t+d}$ by referring to the discussion of heuristic momentum ~\cite{liang2017activity,liang2022improving}.
We compared different $d=2,5,10,20$ in Fig.~\ref{fig:NAG-RGard-Restart}.
As $d$ increases, the convergence is improved.
According to the average speed in (\ref{eq:NAG-average-speed}), when $d$ increases, the slower $\eta_{t}$ grows, and the longer it stays around the optimal parameter, the lower the convergence rate $\rho$ is.

\begin{figure}[!h]
\centering
\includegraphics[width=0.75\textwidth]{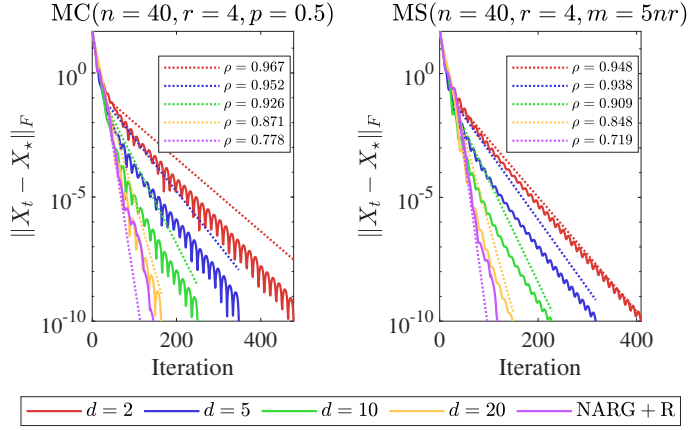}
\caption{Convergence comparison of lazy strategy with different $d$.}\label{fig:NAG-RGard-Restart}
\end{figure}

\subsection{Estimation of Spectral Radius}

Below we verify the spectral radius estimation of the iterative matrix in conclusion, which is the key to the convergence analysis for all algorithms.
We plot the spectral radius of MC and MS under different step sizes and different numbers of observations in Fig.~\ref{fig:rhoH}.
The coincidence of actual spectral radius (solid lines) and its estimation (dashed lines) verifies that the relationship (\ref{eq:IHT-rho}) holds.
There are some differences in stepsize in (\ref{eq:IHT-step}) between MC and MS.
As the number of observations increases, the optimal stepsize $\mu_{\dagger}$ and upper bound $\mu_{\ddagger}$ for MS increase, while the results for MS are reversed.
But, there are obvious upper bounds $\mu_{\ddagger}<1$ for MS and $\mu_{\ddagger}<2$ for MC.

\begin{figure}[!h]
\centering
\includegraphics[width=0.75\textwidth]{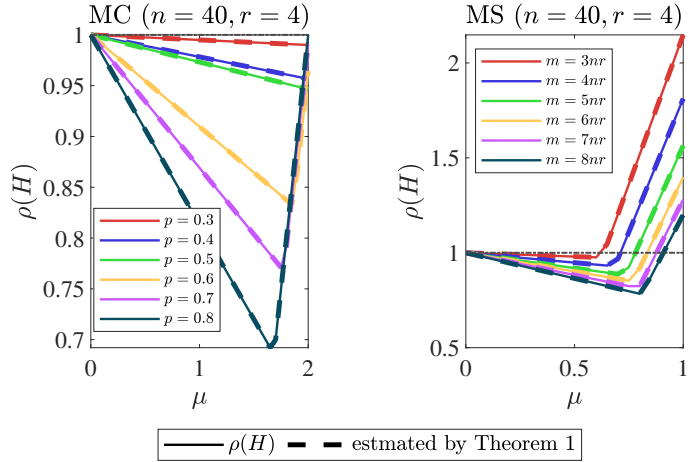}
\caption{Spectral radius estimation of Algorithm~\ref{Alg-IHT}.}\label{fig:rhoH}
\end{figure}

Fig.~\ref{fig:Surf-1} shows that there is a complex 3-D relationship between the spectral radius of the iterative matrix in (\ref{eq:iter-T}) and parameter pair $(\mu,\eta)$.
For a more intuitive presentation, we give its contour in Fig.~\ref{fig:Surf-2}, which verifies our solution to the equation (\ref{eq:quadratic-eqnarray}).
It can be seen that there is an apparent intersection (green dashed line) between the surfaces $\Pi_1$ and $\Pi_2$.
Moreover, the global minimum is on the junction of $\Delta\geq0$ (pink area) and the green dash.
We accurately label the estimates of the optimal parameters with a circle according to Theorem~\ref{th:MatNAG}.
The yellow line corresponding to $\mu>\mu_{\dagger}$ means that introducing momentum does not improve algorithm convergence.

\begin{figure}[!h]
\centering
\includegraphics[width=0.75\textwidth]{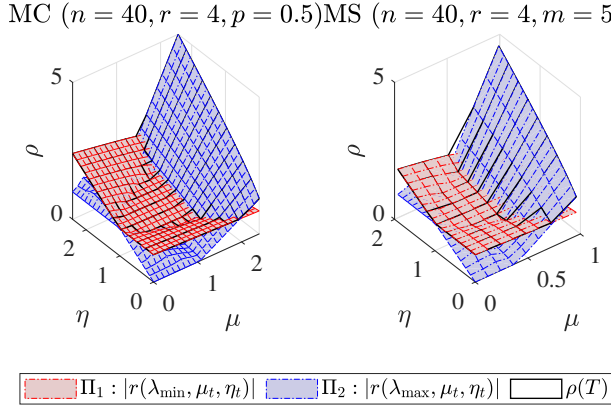}
\caption{3D surface relationship between spectral radius and $(\mu,\eta)$.}\label{fig:Surf-1}
\end{figure}
\begin{figure}[!h]
\centering
\includegraphics[width=0.75\textwidth]{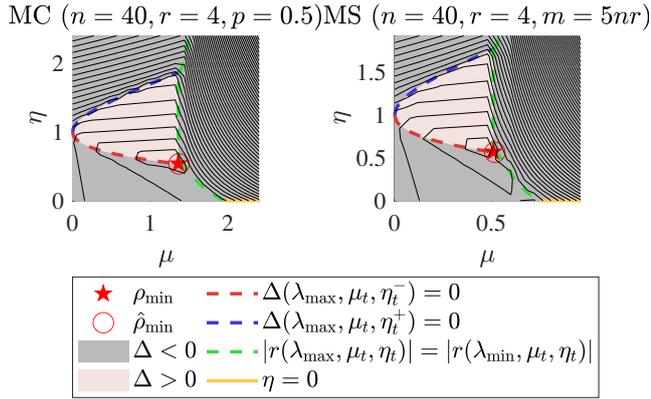}
\caption{2D Top view of spectral radius w.r.t. $(\mu,\eta)$.}\label{fig:Surf-2}
\end{figure}

In addition, we also plot the relationship between the spectral radius individually w.r.t. $\mu$ or $\eta$, respectively.
Fig.~\ref{fig:Surf-3} verifies the staged estimates of spectral radius in Appendix~\ref{app:th:MatNAG}.
The blue line of Fig.~\ref{fig:Surf-3} is equivalent to the yellow line of Fig.~\ref{fig:Surf-2} and the rising part of Fig.~\ref{fig:rhoH}.
The spectral radius of $\mu_{\dagger}$ and $\mu_{\flat}$ correspond to the optimal convergence rates of the accelerated and original algorithms.

\begin{figure}[!h]
\centering
\includegraphics[width=0.75\textwidth]{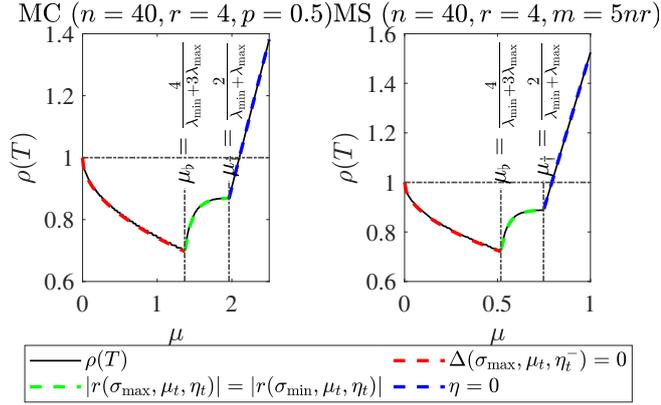}
\caption{Spectral radius estimation of Algorithm~\ref{Alg-NAG}.}\label{fig:Surf-3}
\end{figure}

\subsection{\added{Runtime comparison for larger simulations}}

Under the setting of different parameters $(n,r,p)$, the proposed NARG+R algorithm is compared with the state-of-the-art algorithms, including RGard and ScaledGD~\cite{tong2021accelerating}.
Table~\ref{table:large-scale} presents the average number of iterations and runtime over 20 random simulations with the stopping condition $\|\boldsymbol X_t-\boldsymbol X_\star\|_F\leq10^{-8}$.
By contrast, the matrix size does not affect the number of iterations of the three algorithms.
When $p$ is larger and $r$ is smaller, algorithms for MC usually converge faster.
Similarly, the larger $p$ is, the faster the algorithm for MS.
The impact of the parameter $(p,r)$ will be discussed later.
Overall, the proposed NARG+R has a significant advantage in the number of iterations, which shows that it is competitive in large-scale matrix applications.

\begin{table}[!h]
\caption{Average number of iterations and average runtime (seconds) over 20 random simulations.}
\label{table:large-scale}
\centering
\renewcommand\arraystretch{1.2}
\begin{tabular}{|c|cc|cc|cc|cc|}
\hline
\textbf{Algorithm} & \textbf{Iter} & \textbf{Time} & \textbf{Iter} & \textbf{Time} & \textbf{Iter} & \textbf{Time} & \textbf{Iter} & \textbf{Time} \\
\hline\hline
\multirow{2}{*}{} & \multicolumn{8}{c|}{MC, $n=2500$, sample size $n^2p$}\\
\cline{2-9}
&\multicolumn{2}{c|}{$r=0.05n,p=0.4$} & \multicolumn{2}{c|}{$r=0.05n,p=0.6$} &\multicolumn{2}{c|}{$r=0.1n,p=0.4$} & \multicolumn{2}{c|}{$r=0.1n,p=0.6$}\\
\hline
RGrad&
72.7&13.407&
38.9&7.756&
198&54.03&
69&19.848\\
Scaled GD&
145&25.508&
92&17.27&
360.7&93.906&
157&42.814\\
NARG+R &
56&11.11&
33.25&7.092&
123.5&35.56&
55&16.62\\
\hline\hline
\multirow{2}{*}{} & \multicolumn{8}{c|}{MC, $n=5000$, sample size $n^2p$}\\
\cline{2-9}
&\multicolumn{2}{c|}{$r=0.05n,p=0.4$} & \multicolumn{2}{c|}{$r=0.05n,p=0.6$} &\multicolumn{2}{c|}{$r=0.1n,p=0.4$} & \multicolumn{2}{c|}{$r=0.1n,p=0.6$}\\
\hline
RGrad&
74&69.898&
40&40.418&
203&323.2&
71&117.31\\
Scaled GD&
149&136.28&
95&91.996&
371.45&560.61&
162&253.78\\
NARG+R &
59&59.225&
35&37.25&
129&212.69&
55&93.698\\
\hline\hline
\multirow{2}{*}{} & \multicolumn{8}{c|}{MS, $n=100$, sample size $pnr$}\\
\cline{2-9}
&\multicolumn{2}{c|}{$r=0.05n,p=4$} & \multicolumn{2}{c|}{$r=0.05n,p=6$} &\multicolumn{2}{c|}{$r=0.1n,p=4$} & \multicolumn{2}{c|}{$r=0.1n,p=6$}\\
\hline
RGrad&
319.1&6.7832&
135.8&4.4625&
303.55&12.821&
131.75&8.237\\
Scaled GD&
425.7&6.3149&
207.4&4.7171&
404.7&11.8&
200.5&8.7509\\
NARG+R &
144.8&3.0753&
81.7&2.6768&
136.55&5.7343&
89.7&5.595\\
\hline\hline
\multirow{2}{*}{} & \multicolumn{8}{c|}{MS, $n=200$, sample size $pnr$}\\
\cline{2-9}
&\multicolumn{2}{c|}{$r=0.05n,p=4$} & \multicolumn{2}{c|}{$r=0.05n,p=6$} &\multicolumn{2}{c|}{$r=0.1n,p=4$} & \multicolumn{2}{c|}{$r=0.1n,p=6$}\\
\hline
RGrad&
337.6&71.826&
142.95&45.913&
320.2&135.86&
138.8&88.308\\
Scaled GD&
450.45&63.297&
217.35&45.782&
425.8&118.98&
210.4&88.039\\
NARG+R &
154.75&32.835&
80.6&25.718&
171.95&72.79&
81&51.26\\
\hline
\end{tabular}
\end{table}

\subsection{\added{Spectral initialization versus Random initialization}}

Taking RGrad as an example, we analyze the impact of initialization, and its settings are shown in Table~\ref{table:parameter}.
Random initialization adds Gaussian noise with different variances based on spectral initialization~\cite{chi2019nonconvex}.
The comparisons are shown in Fig.~\ref{fig:initialization}.
The results show that random initialization has a more significant impact on MC.
As $\sigma$ increases, random initialization moves the unobserved further away from the optimal solution, which makes the algorithm more challenging to satisfy (\ref{eq:Basin-of-Attraction}).
In contrast, spectral initialization speeds up the process, dramatically improving local search efficiency.
On the other hand, MS is less affected by initialization, and linear convergence requires only a few iterations.
Once condition (\ref{eq:Basin-of-Attraction}) holds, the local convergence rate is independent of initialization and is related to the spectral radius.
\begin{table*}[!h]
	\caption{Initialization settings.}
	\label{table:parameter}
	\centering
	\begin{tabular}{c|c|c} 
		\hline
& \textbf{Spectral initialization} & \textbf{Random initialization} \\ 
\hline
\textbf{MC}
& $\boldsymbol X_0=\mathcal P_r(\frac{1}{p}\mathcal P_{\Omega}(\boldsymbol X_{\mathsf{ob}}))$ & \thead{$\boldsymbol X_0=\mathcal P_r(\frac{1}{p}\mathcal P_{\Omega}(\boldsymbol X_{\mathsf{ob}})+\mathcal P_{\bar\Omega}(\boldsymbol Y_{\mathsf{rand}}))$\\ $[\boldsymbol Y_{\mathsf{rand}}]_{i,j}\sim\mathcal N(0,\sigma^2)$} \\ 
\hline
\textbf{MS}
 & $\boldsymbol X_0=\mathcal P_r(\mathcal A^*(\boldsymbol y_{\mathsf{ob}}))$ &\thead{$\boldsymbol X_0=\mathcal P_r(\mathcal A^*(\boldsymbol y_{\mathsf{ob}}+\boldsymbol y_{\mathsf{rand}}))$\\$[\boldsymbol y_{\mathsf{rand}}]_i\sim\mathcal N(0,\sigma^2)$}  \\ 
\hline
	\end{tabular}
\end{table*}

\begin{figure}[!h]
\centering
\includegraphics[width=0.75\textwidth]{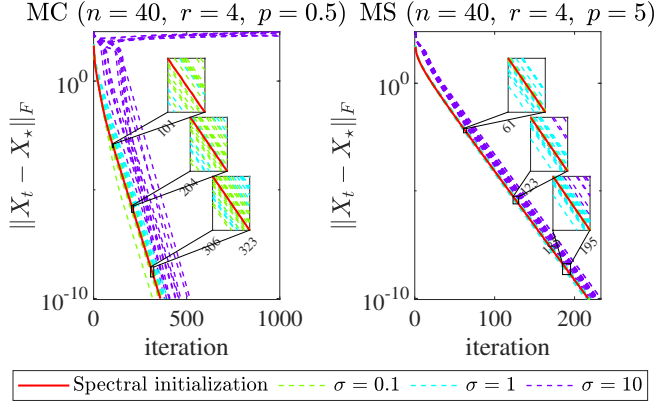}
\caption{Comparison of Spectral initialization and Random initialization with different $\sigma$.}\label{fig:initialization}
\end{figure}

\subsection{\added{Numerical phase transition}}
\label{sec:PhaseTransition}
Finally, we evaluate and compare the recovery rate of NAGR+R with ScaledGD and RGrad.
If $\|\hat {\boldsymbol X}-\boldsymbol X_\star\|_F\leq 10^{-3}$, we judge it as a successful recovery, where $\hat {\boldsymbol X}$ is the output of the algorithm.
The empirical success rate was calculated by repeating 20 trials with different ranks and sample sizes.
For sample size, we use sampling rate $p=|\Omega|/n^2$ for MC and $m=pnr$ for MS.
The empirical phase transitions are presented in Fig.~\ref{fig:phasetrans}, where white indicates successful recovery and black indicates failure for NARG+R.
Our algorithm produces a more extensive white area on both tasks than the others.
The theoretical lower bound for estimating sample size using the iterative matrix is still an open problem, and we leave it as future work.

\begin{figure}[!h]
\centering
\includegraphics[width=0.75\textwidth]{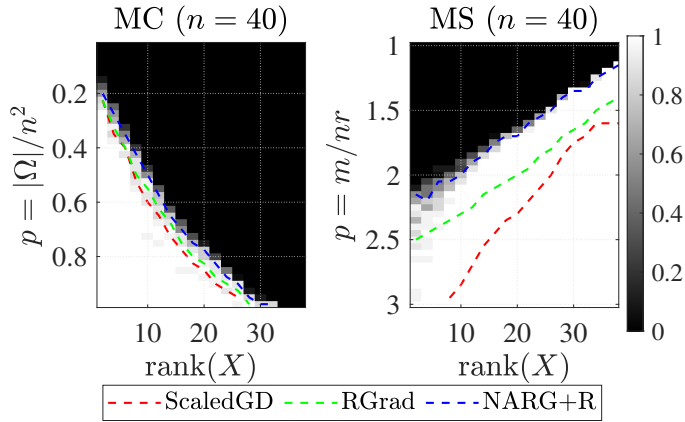}
\caption{Phase transition of NARG+R for MC and MS over 20 random simulations. The dotted line indicates that the success rate of all methods reaches 50\%, which serves as a reference for comparison.}\label{fig:phasetrans}
\end{figure}



\section{Conclusion}
\label{sec:conclusion}

We have proposed a novel efficient Nesterov's Accelerated Riemannian Gradient for the low-rank matrix estimation problem.
To our knowledge, this is the first work to connect manifold and tangent space through orthographic retraction and its inverse.
As the name suggests, it inherits the low computational complexity of the Riemannian Gradient and the fast linear convergence rate of Nesterov's Accelerated Gradient.
The spectral radius estimates the local convergence rate.
The algorithm matches the theoretical optimal rate based on the adaptive restart scheme.
Numerical simulations of both MS and MC illustrate that our algorithm is superior in computational complexity.
It would be interesting to study theoretical optimal sample complexity by the spectral radius of the iterative matrix.
Another further direction is the generalization to tensors, such as Tucker decomposition and Tensor-Train decomposition.



%
%
%

\appendix

\section{Auxiliary lemmas}

\subsection{Relationship of matrix eigenvalues}

\begin{lemma}\label{lem:eigenvalue-part1}
Let $\boldsymbol \Theta$ be a symmetric positive semi-definite matrix, and $\boldsymbol P\in\mathbb R^{n\times n}$ be an orthogonal projection matrix.
Denote $\boldsymbol P^{\perp}=\boldsymbol I-\boldsymbol P$, then there exists an eigenvalue \(\lambda\neq0\) of \((\boldsymbol I-\mu \boldsymbol \Theta ) \boldsymbol P^\perp\), such that
\begin{eqnarray*}
\lambda_{i}(\mu \boldsymbol \Theta \boldsymbol P^\perp +\boldsymbol P)=\lambda_{i}(\mu \boldsymbol \Theta \boldsymbol P^\perp)=1-\lambda.
\end{eqnarray*}
\end{lemma}

\begin{proof}
Assume \(\text{rank}(\boldsymbol P)=r\).
According to idempotent, we get the eigenvalues of \(\boldsymbol P\) and \(\boldsymbol P^{\perp}\) of the form.
\begin{eqnarray*}
\lambda(\boldsymbol P)=\{\underbrace {1,\ldots, 1}_{r},\underbrace {0,\ldots, 0}_{n-r}\},\lambda(\boldsymbol P^{\perp})=\{\underbrace {1,\ldots, 1}_{n-r},\underbrace {0,\ldots, 0}_{r}\}.
\end{eqnarray*}
Here, we use \(\boldsymbol u_i\) and \(\boldsymbol v_j\) represent the eigenvectors of \(\boldsymbol P\) corresponding to eigenvalues \(1\) and \(0\), respectively.
From the orthogonal relation between \(\boldsymbol P\) and \(\boldsymbol P^{\perp}\), $\boldsymbol P\boldsymbol u_i=\boldsymbol u_i,\boldsymbol P\boldsymbol v_j=\boldsymbol 0,\boldsymbol P^\perp \boldsymbol u_i=\boldsymbol 0,\boldsymbol P^\perp \boldsymbol v_j=\boldsymbol v_j$.
Further, we have
\begin{eqnarray}\label{eq:lambda=0}
(\mu \boldsymbol \Theta \boldsymbol P^\perp +\boldsymbol P)\boldsymbol u_i=\boldsymbol u_i,
(\mu \boldsymbol \Theta \boldsymbol P^\perp )\boldsymbol u_i=\boldsymbol 0, 
(\mu \boldsymbol \Theta \boldsymbol P^\perp +\boldsymbol P)\boldsymbol v_j=(\mu \boldsymbol \Theta \boldsymbol P^\perp)\boldsymbol v_j=\mu \boldsymbol \Theta \boldsymbol v_j.
\end{eqnarray}
It can be known that \(\boldsymbol u_i\) corresponds to the eigenvector of \(\mu \boldsymbol \Theta \boldsymbol P^\perp +\boldsymbol P\) with the eigenvalue of \(1\) and the eigenvector of \(\mu \boldsymbol \Theta \boldsymbol P^\perp\) with the eigenvalue of \(0\), respectively.
Besides, if \(\boldsymbol v_j\) happens to be an eigenvector of \(\mu \boldsymbol \Theta\), then \(\boldsymbol v_j\) is also an eigenvector of \(\mu \boldsymbol \Theta \boldsymbol P^\perp +\boldsymbol P\) and \(\mu \boldsymbol \Theta \boldsymbol P^\perp\).
This conjecture implies the relevance of the above three matrix eigendecompositions.
To this end, assume that there exists a non-zero vector \(\boldsymbol x\) such that
\begin{eqnarray*}
(\boldsymbol I-\mu \boldsymbol \Theta ) \boldsymbol P^\perp \boldsymbol x=\lambda \boldsymbol x,
\end{eqnarray*}
then
\begin{eqnarray*}
(\mu \boldsymbol \Theta \boldsymbol P^\perp +\boldsymbol P)\boldsymbol x=(1-\lambda)\boldsymbol x.
\end{eqnarray*}
For \(\lambda\neq0\), we will discuss \(\boldsymbol P\boldsymbol x=(1-\lambda)\boldsymbol x-\mu \boldsymbol \Theta \boldsymbol P^\perp \boldsymbol x\) case by case:

Case 1:
when \(\boldsymbol P\boldsymbol x=\boldsymbol 0\), i.e., \(\boldsymbol x\) is a linear combination of \(\boldsymbol v_i\).
Obviously, \(\mu \boldsymbol \Theta \boldsymbol P^\perp x=(1-\lambda)\boldsymbol x\) holds.
We obtain \(1-\lambda\) is the eigenvalue of matrix \(\mu \boldsymbol \Theta \boldsymbol P^\perp\).

Case 2:
when \(\boldsymbol P\boldsymbol x\neq\boldsymbol 0\), we know that \(\boldsymbol x\) can always be represented as a linear combination of orthonormal bases, as follows
\begin{eqnarray*}
\boldsymbol x=\sum_{i=1}^r\alpha_i \boldsymbol u_i+\sum_{j=1}^{n-r}\beta_j \boldsymbol v_j,
\end{eqnarray*}
and \(\boldsymbol P\boldsymbol x\neq\boldsymbol 0\) means that there exists \(\alpha_i\neq 0\), otherwise \(\boldsymbol P\boldsymbol x=\sum_{j=1}^{n-r}\beta_j \boldsymbol P\boldsymbol v_j=\boldsymbol 0\) if \(\forall i,\alpha_i=0\).
And we expand the formula to get
\begin{eqnarray*}
(\boldsymbol I-\mu \boldsymbol \Theta ) \sum_{j=1}^{n-r}\beta_j \boldsymbol v_j=(\boldsymbol I-\mu \boldsymbol \Theta ) \boldsymbol P^\perp \boldsymbol x=\lambda \boldsymbol x=\lambda (\sum_{i=1}^r\alpha_i \boldsymbol u_i+\sum_{j=1}^{n-r}\beta_j \boldsymbol v_j),
\end{eqnarray*}
where the left-hand side is a linear representation of the basis vector \(\{\boldsymbol v_j\}\), while the right-hand side is a linear combination of mutually orthogonal basis vectors \(\{\boldsymbol u_i\}\) and \(\{\boldsymbol v_j\}\).
So when \(\lambda\neq0\), \(\forall i,\alpha_i=0\) holds.
This contradicts \(\boldsymbol P\boldsymbol x\neq\boldsymbol 0\).
\end{proof}


\subsection{Perturbation Analysis of Subspaces}

\begin{lemma}[Wedin's $\sin\Theta$ Theorem \cite{chen2021spectral}]
\label{lem:Wedin}
Let $\boldsymbol X_t=\boldsymbol U_t \boldsymbol \Sigma_t \boldsymbol V_t^\top$ and $\boldsymbol X_\star=\boldsymbol U_\star \boldsymbol \Sigma_\star \boldsymbol V_\star^\top$ be the SVD of $\boldsymbol X_t, \boldsymbol X_\star\in\mathbb M_r$, respectively.
If $\|\boldsymbol X_t-\boldsymbol X_\star\|<\sigma_r(\boldsymbol X_\star)$, there is an upper bound for the perturbation of the singular subspace as follows
\begin{eqnarray*}
\max\{\|P_{\boldsymbol U_t}^\perp-P_{\boldsymbol U_\star}^\perp\|,\|P_{\boldsymbol V_t}^\perp-P_{\boldsymbol V_\star}^\perp\|\}\leq\frac{2\|\boldsymbol X_t-\boldsymbol X_\star\|}{\sigma_r(\boldsymbol X_\star)}.
\end{eqnarray*}
\end{lemma}

\begin{lemma}[Perturbation of subspace projection~\cite{wei2016guarantees}]
\label{lem:subspace project-error}
Let $\boldsymbol X_t=\boldsymbol U_t \boldsymbol \Sigma_t \boldsymbol V_t^\top$ and $\boldsymbol X_\star=\boldsymbol U_\star \boldsymbol \Sigma_\star \boldsymbol V_\star^\top$ be the SVD of $\boldsymbol X_t, \boldsymbol X_\star\in\mathbb M_r$, respectively.
If $\|\boldsymbol X_t-\boldsymbol X_\star\|<\sigma_r(\boldsymbol X_\star)$, then the following inequality is satisfied
\begin{eqnarray*}
\|P_{\boldsymbol U_\star}^\perp \boldsymbol X_t P_{\boldsymbol V_\star}^\perp\|_F\leq \frac{\|\boldsymbol X_t-\boldsymbol X_\star\|_F^2}{\sigma_r(\boldsymbol X_\star)}.
\end{eqnarray*}
\end{lemma}

\section{Proof of Theorem~\ref{th:MatIHT}}
\label{app:Th:MatIHT}
\begin{proof}
Let the residual matrix \(\boldsymbol E_t=\boldsymbol X_t-\boldsymbol X_\star\).
According to the iteration, we have
\begin{eqnarray}\label{eq:iterate-residual}
\begin{aligned}
\boldsymbol E_{t+1}&=\boldsymbol X_{t+1}-\boldsymbol X_\star\\
&=\mathcal P_r(\boldsymbol X_{t}-\mu_t\nabla f(\boldsymbol X_t))-\boldsymbol X_\star\\
&=\mathcal P_r(\boldsymbol X_\star +\boldsymbol X_{t}-\boldsymbol X_\star-\mu_t\nabla f(\boldsymbol X_t))-\boldsymbol X_\star\\
&=\mathcal P_r(\boldsymbol X_\star +\boldsymbol E_{t}-\mu_t\nabla f(\boldsymbol X_t))-\boldsymbol X_\star\\
&\stackrel{(a)}{=}(\boldsymbol E_{t}-\mu_t\nabla f(\boldsymbol X_t))-P_{\boldsymbol U_\star}^\perp(\boldsymbol E_{t}-\mu_t\nabla f(\boldsymbol X_t))P_{\boldsymbol V_\star}^\perp+\mathcal O(\|\boldsymbol E_{t}\|_F^2),\\
\end{aligned}
\end{eqnarray}
where \((a)\) is the first-order expansion (\ref{eq:TSVD-expansion}) of the truncated SVD. 
\added{Since $\|\mathcal I-\mu \mathcal A^*\mathcal A\|\leq 1$, we have $\|\boldsymbol E_{t}-\mu\nabla f(\boldsymbol X_t)\|_F=\|(\mathcal I-\mu_t\mathcal A^*\mathcal A)(\boldsymbol E_t)\|_F\leq\|(\boldsymbol E_t)\|_F\leq\sigma_{\replaced{r}{\min}}(\boldsymbol X_\star)/2$, which verifies the condition of Lemma~\ref{lem:TSVD-Perturbation} holds.}
After vectorizing \(\boldsymbol e_{t+1}=\text{vec}(\boldsymbol E_t)\), we get
\begin{eqnarray*}
\begin{aligned}
\boldsymbol e_{t+1}
&=\text{vec}((\boldsymbol E_{t}-\mu_t\nabla f(\boldsymbol X_t))-P_{\boldsymbol U_\star}^\perp(\boldsymbol E_{t}-\mu_t\nabla f(\boldsymbol X_t))P_{\boldsymbol V_\star}^\perp)+\mathcal O(\|\boldsymbol E_t\|_F^2)\\
&\stackrel{(a)}{=}(\boldsymbol I-P_{\boldsymbol V_\star}^\perp\otimes P_{\boldsymbol U_\star}^\perp)\text{vec}(\boldsymbol E_{t}-\mu_t\nabla f(\boldsymbol X_t))+\mathcal O(\|\boldsymbol E_t\|_F^2)\\
&\stackrel{(b)}{=}\underbrace{(\boldsymbol I-P_{\boldsymbol V_\star}^\perp\otimes P_{\boldsymbol U_\star}^\perp)(\boldsymbol I-\mu_t \boldsymbol \Theta)}_{\boldsymbol H(\mu_t)}\boldsymbol e_{t}+\mathcal O(\|\boldsymbol e_{t}\|_2^2),\\
\end{aligned}
\end{eqnarray*}
where \((a)\) uses vectorization of Kronecker product, i.e., \(\text{vec}(\boldsymbol A\boldsymbol B\boldsymbol C)=(\boldsymbol C^\top \otimes \boldsymbol A)\text{vec}(\boldsymbol B)\).
\((b)\) is based on (\ref{eq:vec-grad}) and \(\|\boldsymbol E_t\|_F=\|\boldsymbol e_{t}\|_2\).
The convergence rate with the constant stepsize \(\mu_t\equiv\mu\) is determined by the spectral radius of the matrix \(\boldsymbol H=\boldsymbol H(\mu)\).
\begin{eqnarray*}
\rho(\boldsymbol H)=\max_\lambda |\lambda_i(\boldsymbol H)|=\max{(|\lambda_{\max}(\boldsymbol H)|,|\lambda_{\min}(\boldsymbol H)|)}.
\end{eqnarray*}
Thus, the maximum and minimum eigenvalues of \(\boldsymbol H\) should be compared.
Taking MS as an example, we compute the largest eigenvalue.
\begin{eqnarray*}
\begin{aligned}
\lambda_{\max}(\boldsymbol H_{\mathsf{MS}})
&=1-\lambda_{\min}(\mu \boldsymbol \Theta+P_{\boldsymbol V_\star}^\perp\otimes P_{\boldsymbol U_\star}^\perp-\mu  (P_{\boldsymbol V_\star}^\perp\otimes P_{\boldsymbol U_\star}^\perp)\boldsymbol \Theta)\\
&\stackrel{(a)}{=}1-\lambda_{\min}(\mu \boldsymbol \Theta-\mu  (P_{\boldsymbol V_\star}^\perp\otimes P_{\boldsymbol U_\star}^\perp)\boldsymbol \Theta)\\
&=1-\mu\lambda_{\min}((\boldsymbol I-P_{\boldsymbol V_\star}^\perp\otimes P_{\boldsymbol U_\star}^\perp)\boldsymbol \Theta),\\
\end{aligned}
\end{eqnarray*}
where \((a)\) is based on Lemma~\ref{lem:eigenvalue-part1}.
Similarly, the minimum eigenvalue results are as follows:
\begin{eqnarray*}
\lambda_{\min}(\boldsymbol H_{\mathsf{MS}})=1-\mu\lambda_{\max}((\boldsymbol I-P_{\boldsymbol V_\star}^\perp\otimes P_{\boldsymbol U_\star}^\perp)\boldsymbol \Theta).
\end{eqnarray*}
Obviously, the optimal spectral radius occurs when \(\lambda_{\max}(\boldsymbol H_{\mathsf{MS}})=-\lambda_{\min}(\boldsymbol H_{\mathsf{MS}})\), i.e., $1-\mu\lambda_{\min}=\mu\lambda_{\max}-1$.
The corresponding stepsize is $\mu_\dagger=\frac{2}{\lambda_{\min}+\lambda_{\max}}$.
\added{Due to $\boldsymbol I-P_{\boldsymbol V_\star}^\perp\otimes P_{\boldsymbol U_\star}^\perp$ is orthogonal projector, we have $\|\boldsymbol I-P_{\boldsymbol V_\star}^\perp\otimes P_{\boldsymbol U_\star}^\perp\|=1$.
It is easy to check}
\begin{eqnarray*}
\rho(\boldsymbol H)\leq \|\boldsymbol I-P_{\boldsymbol V_\star}^\perp\otimes P_{\boldsymbol U_\star}^\perp\| \|\boldsymbol I-\mu_t \boldsymbol \Theta\|\leq \|\boldsymbol I-\mu_t \boldsymbol \Theta\| \leq 1,
\end{eqnarray*}
\added{so} we can get (\ref{eq:IHT-rho}).
Especially for MC, as shown in (\ref{eq:sample-matrix}),
we get a simplified result similar to~\cite{vu2019accelerating}.
\begin{eqnarray*}
\begin{aligned}
\lambda_{\max}(\boldsymbol H_{\mathsf{MC}})
&\stackrel{(a)}{=}1-\mu\lambda_{\min}(  \boldsymbol S_{\Omega}^\top(I-P_{\boldsymbol V}^\perp\otimes P_{\boldsymbol U}^\perp)\boldsymbol S_{\Omega})
\stackrel{(b)}{=}1-\mu\lambda_{\min}( \boldsymbol I- \boldsymbol S_{\Omega}^\top(P_{\boldsymbol V_\star}^\perp\otimes P_{\boldsymbol U_\star}^\perp)\boldsymbol S_{\Omega})\\
&=1-\mu(1-\lambda_{\max}( \boldsymbol S_{\Omega}^\top(P_{\boldsymbol V_\star}^\perp\otimes P_{\boldsymbol U_\star}^\perp)\boldsymbol S_{\Omega}))
\stackrel{(c)}{=}1-\mu(1-\lambda_{\max}(\boldsymbol S_{\Omega}\boldsymbol S_{\Omega}^\top(P_{\boldsymbol V_\star}^\perp\otimes P_{\boldsymbol U_\star}^\perp)))\\
&\stackrel{(d)}{=}1-\mu(\lambda_{\min}(P_{\boldsymbol V_\star}^\perp\otimes P_{\boldsymbol U_\star}^\perp-\boldsymbol S_{\Omega}\boldsymbol S_{\Omega}^\top(P_{\boldsymbol V_\star}^\perp\otimes P_{\boldsymbol U_\star}^\perp)))
\stackrel{(e)}{=}1-\mu(\lambda_{\min}(\boldsymbol S_{\bar\Omega}\boldsymbol S_{\bar\Omega}^\top(P_{\boldsymbol V_\star}^\perp\otimes P_{\boldsymbol U_\star}^\perp)))\\
&\stackrel{(f)}{=}1-\mu(\lambda_{\min}(\boldsymbol S_{\bar\Omega}^\top(P_{\boldsymbol V_\star}^\perp\otimes P_{\boldsymbol U_\star}^\perp)\boldsymbol S_{\bar\Omega}))
=1-\mu(\sigma_{\min}^2(\boldsymbol S_{\bar\Omega}^\top(\boldsymbol V_{_\star\perp} \otimes \boldsymbol U_{_\star\perp}))),
\end{aligned}
\end{eqnarray*}
where \((a)\), \((c)\) and \((f)\) are based on the fact that $AB$ and $BA$ 
have the same eigenvalues, \((b)\) and \((e)\) correspond to the properties of the sampling matrix in (\ref{eq:sample-matrix}), \((d)\) uses Lemma~\ref{lem:eigenvalue-part1}.
Similarly, the minimum eigenvalue results are as follows
\begin{eqnarray*}
\lambda_{\min}(\boldsymbol H_{\mathsf{MC}})=1-\mu(\sigma_{\max}^2(\boldsymbol S_{\bar\Omega}^\top(\boldsymbol V_{_\star\perp}\otimes \boldsymbol U_{_\star\perp}))).
\end{eqnarray*}
We can estimate the convergence rate of Algorithm~\ref{Alg-IHT}.
\end{proof}

\section{Proof of Proposition~\ref{prop:NIHT-Convergence}}
\label{app:prop:NIHT-Convergence}

\begin{proof}
Vectorizing (\ref{eq:proj-gradient}) yields $\nabla_\mathcal R f(\boldsymbol x_t)=\boldsymbol P\nabla f(\boldsymbol x_t)$, where $\boldsymbol P=(\boldsymbol I-P_{\boldsymbol U}^\perp\otimes P_{\boldsymbol V}^\perp)$ is the orthogonal projection matrix.
Bring (\ref{eq:mu-close-form}) into the loss function to get
\begin{eqnarray*}
\begin{aligned}
f(\boldsymbol x_{t+1})&=\frac{1}{2}(\boldsymbol x_{t+1}-\boldsymbol x_\star)^\top \boldsymbol \Theta (\boldsymbol x_{t+1}-\boldsymbol x_\star)\\
&=\frac{1}{2}(\boldsymbol x_t -\mu_t \nabla_\mathcal R f(\boldsymbol x_t)-\boldsymbol x_\star)^\top \boldsymbol \Theta (\boldsymbol x_t -\mu_t \nabla_\mathcal R f(\boldsymbol x_t)-\boldsymbol x_\star)+\mathcal O(\|\boldsymbol x_t -\boldsymbol x_\star\|_2^2)\\
&=f(\boldsymbol x_t)-\mu_t\nabla_\mathcal R f(\boldsymbol x_t)^\top \nabla f(\boldsymbol x_t)+\frac{\mu_t^2}{2}\nabla_\mathcal R f(\boldsymbol x_t)^\top \boldsymbol \Theta \nabla_\mathcal R f(\boldsymbol x_t)+\mathcal O(\|\boldsymbol x_t -\boldsymbol x_\star\|_2^2)\\
&=f(\boldsymbol x_t)-\frac{(\nabla_\mathcal R f(\boldsymbol x_t)^\top \nabla f(\boldsymbol x_t))^2}{2\nabla_\mathcal R f(\boldsymbol x_t)^\top \boldsymbol \Theta \nabla_\mathcal R f(\boldsymbol x_t)}+\mathcal O(\|\boldsymbol x_t -\boldsymbol x_\star\|_2^2)\\
&\stackrel{(a)}{=}\left(1-\frac{(\nabla_\mathcal R f(\boldsymbol x_t)^\top \nabla f(\boldsymbol x_t))^2}{(\nabla f(\boldsymbol x_t)^\top (\boldsymbol P\boldsymbol \Theta \boldsymbol P) \nabla f(\boldsymbol x_t))(\nabla_\mathcal R f(\boldsymbol x_t)^\top (\boldsymbol P\boldsymbol \Theta \boldsymbol P)^{+} \nabla_\mathcal R f(\boldsymbol x_t))}\right)f(\boldsymbol x_t)+\mathcal O(\|\boldsymbol x_t -\boldsymbol x_\star\|_2^2)\\
&\stackrel{(b)}{\leq}\left(1-\frac{4}{\frac{\lambda_{\max}(\boldsymbol P\boldsymbol \Theta \boldsymbol P)}{\lambda_{\min}(\boldsymbol P\boldsymbol \Theta \boldsymbol P)} + 2+\frac{\lambda_{\min}(\boldsymbol P\boldsymbol \Theta \boldsymbol P)}{\lambda_{\max}(\boldsymbol P\boldsymbol \Theta \boldsymbol P)}}\right)f(\boldsymbol x_t)+\mathcal O(\|\boldsymbol x_t -\boldsymbol x_\star\|_2^2)\\
&\stackrel{(c)}{\leq}\left(\frac{\kappa-1}{\kappa+1}\right)^2 f(\boldsymbol x_t)+\mathcal O(\|\boldsymbol x_t -\boldsymbol x_\star\|_2^2),\\
\end{aligned}
\end{eqnarray*}
where $(a)$ is because of $f(\boldsymbol x)=\frac{1}{2}(\boldsymbol x-\boldsymbol x_\star)^\top \boldsymbol \Theta (\boldsymbol x-\boldsymbol x_\star)=\frac{1}{2}\nabla_\mathcal R f(\boldsymbol x)^\top (\boldsymbol P\boldsymbol \Theta \boldsymbol P)^{+} \nabla_\mathcal R f(\boldsymbol x_t)$.
Furthermore, since $\frac{|\nabla_\mathcal R f(\boldsymbol x_t)^\top \nabla f(\boldsymbol x_t)|}{\|\nabla_\mathcal R f(\boldsymbol x_t)^\top\|_2 \|\nabla f(\boldsymbol x_t)\|_2}\geq0$, we apply the generalized Kantorovich type inequality~\cite{huang2005direct} in Lemma~\ref{lem:Kantorovich} to get $(b)$.
To prove $(c)$, we only need to show
\begin{eqnarray*}
\begin{aligned}
\frac{\lambda_{\max}(\boldsymbol P\boldsymbol \Theta \boldsymbol P)}{\lambda_{\min}(\boldsymbol P\boldsymbol \Theta \boldsymbol P)}
&=\|\boldsymbol P\boldsymbol \Theta \boldsymbol P\|\|(\boldsymbol P\boldsymbol \Theta)^+\boldsymbol P^+\|
\leq\|\boldsymbol P\boldsymbol \Theta\| \|\boldsymbol P\|\|\boldsymbol P^+\|\|(\boldsymbol P\boldsymbol \Theta)^+\|\\
&=\|\boldsymbol P\boldsymbol \Theta\| \|(\boldsymbol P\boldsymbol \Theta)^+\|
= \frac{\lambda_{\max}(\boldsymbol P\boldsymbol \Theta)}{\lambda_{\min}(\boldsymbol P\boldsymbol \Theta)}
:=\kappa,
\end{aligned}
\end{eqnarray*}
here, $\lambda_{\min}(\cdot)$ means the smallest non-zero eigenvalue.
\end{proof}

\begin{lemma}[Kantorovich inequality~\cite{huang2005direct}]\label{lem:Kantorovich}
Let $\boldsymbol A$ be a symmetric (semi-) positive definite matrix, and $\lambda_{\max}$ and $\lambda_{\min}$ correspond to the largest and smallest non-zero eigenvalues, respectively.
If $\boldsymbol x,\boldsymbol y\in\mathbb R^n$ satisfies $\frac{|\boldsymbol x^\top \boldsymbol y|}{\|\boldsymbol x\|_2 \|\boldsymbol y\|_2}\geq\cos\theta$ with $0\leq \theta\leq \frac{\pi}{2}$, then
\begin{eqnarray*}
\frac{(\boldsymbol x^\top \boldsymbol y)^2}{(\boldsymbol x^\top \boldsymbol A\boldsymbol x)(\boldsymbol y^\top \boldsymbol A^{+}\boldsymbol y)}\geq\frac{4}{\kappa + 2+\kappa^{-1}},
\end{eqnarray*}
where $\kappa=\frac{\lambda_{\max}}{\lambda_{\min}}\frac{1+\sin\theta}{1-\sin\theta}$ and $(\cdot)^{+}$ is the Moore-Penrose inverse.
When $\boldsymbol A$ is positive definite and $\boldsymbol x=\boldsymbol y$, i.e., $\boldsymbol A^{+}=\boldsymbol A^{-1}$ and $\theta=0$, the above inequality degenerates into the traditional form.
\end{lemma}

\section{Proof of Theorem~\ref{th:MatNAG}}
\label{app:th:MatNAG}

\begin{proof}
According to Algorithm~\ref{Alg-NAG}, we calculate the error as follows.
\begin{eqnarray*}
\begin{aligned}
\boldsymbol E_{t+1}&=\boldsymbol X_{t+1}-\boldsymbol X_\star\\
&=\mathcal P_r(\boldsymbol Y_t-\mu_t \nabla f(\boldsymbol Y_t))-\boldsymbol X_\star\\
&=\mathcal P_r(\boldsymbol X_\star+\boldsymbol Y_t-\boldsymbol X_\star-\mu_t \nabla f(\boldsymbol Y_t))-\boldsymbol X_\star\\
&=(\boldsymbol Y_t-\boldsymbol X_\star-\mu_t \nabla f(\boldsymbol Y_t))-P_{\boldsymbol U_\star}^\perp(\boldsymbol Y_t-\boldsymbol X_\star-\mu_t \nabla f(\boldsymbol Y_t))P_{\boldsymbol V_\star}^\perp+\mathcal O(\|\boldsymbol Y_t-\boldsymbol X_\star\|_F^2).\\
\end{aligned}
\end{eqnarray*}
After vectorizing, we have
\begin{eqnarray*}
\begin{aligned}
\boldsymbol e_{t+1}&=\underbrace{(\boldsymbol I-P_{\boldsymbol V_\star}^\perp\otimes P_{\boldsymbol U_\star}^\perp)(\boldsymbol I-\mu_t \boldsymbol \Theta)}_{\boldsymbol H_t=\boldsymbol H(\mu_t)}\text{vec}(\boldsymbol Y_t-\boldsymbol X_\star)+\mathcal O(\|\boldsymbol Y_t-\boldsymbol X_\star\|_F^2)\\
&=(1+\eta_t)\boldsymbol H_t\boldsymbol e_{t}-\eta_t \boldsymbol H_t\boldsymbol e_{t-1}+\mathcal O(\|\boldsymbol e_{t}\|_2^2).\\
\end{aligned}
\end{eqnarray*}
Stacking the errors of two adjacent iterations, we get the recursive form
\begin{eqnarray*}
\begin{aligned}
\begin{pmatrix}
\boldsymbol e_{t+1}\\\boldsymbol e_{t}
\end{pmatrix}=\underbrace {\begin{pmatrix}
(1+\eta_t)\boldsymbol H_t &-\eta_t\boldsymbol H_t\\
\boldsymbol I&\boldsymbol 0
\end{pmatrix}}_{\boldsymbol T}
\begin{pmatrix}
\boldsymbol e_{t}\\\boldsymbol e_{t-1}
\end{pmatrix}.
\end{aligned}
\end{eqnarray*}
The convergence rate depends on the spectral radius \(\rho(\boldsymbol T)\) of \(\boldsymbol T\in\mathbb R^{2n_1n_2\times 2n_1n_2}\).
According to the eigendecomposition in~\cite{vu2019accelerating}, \(\boldsymbol T\) is similar to the block diagonal matrix composed of the \(2\times 2\) matrix \(\boldsymbol T_j\), i.e., \(\boldsymbol T\sim \text{bldiag}(\boldsymbol T_1,\boldsymbol T_2,\ldots,\boldsymbol T_{n_1n_2})\), where each block \(\boldsymbol T_j\in\mathbb R^{2\times 2}\) is form
\begin{eqnarray*}
\boldsymbol T_j=\begin{pmatrix}
(1+\eta_t)(1-\mu_t\lambda_j) &-\eta_t(1-\mu_t\lambda_j)\\
1&0
\end{pmatrix}.
\end{eqnarray*}
where \(\lambda_j\) is the eigenvalue of matrix \((\boldsymbol I-P_{\boldsymbol V_\star}^\perp\otimes P_{\boldsymbol U_\star}^\perp)\boldsymbol \Theta\).
Next, we aim to find the eigenvalues of the matrix \(\boldsymbol T_j\) using the characteristic polynomial.
\begin{eqnarray}\label{eq:quadratic-eqnarray}
r^2-(1+\eta_t)(1-\mu_t\lambda_j)r+\eta_t(1-\mu_t\lambda_j)=0.
\end{eqnarray}
According to the quadratic formula, set the discriminant \(\Delta (\lambda_j,\mu_t,\eta_t)=(1+\eta_t)^2(1-\mu_t\lambda_j)^2-4\eta_t(1-\mu_t\lambda_j)\), then the solution to (\ref{eq:quadratic-eqnarray}) is:
\begin{eqnarray}\label{eq:solution-eigen}
r^{\pm}(\lambda_j,\mu_t,\eta_t)=\frac{(1+\eta_t)(1-\mu_t\lambda_j)\pm\sqrt{\Delta (\lambda_j,\mu_t,\eta_t)}}{2},
\end{eqnarray}
where the superscript $(\cdot)^{\pm}$ means addition or subtraction in numerator.
For given \(\boldsymbol T\) with fixed \((\mu_t,\eta_t)\), \(\rho(\boldsymbol T)=\max_{\lambda_j} |r^{\pm}(\lambda_j,\mu_t,\eta_t)|\) is continuous and quasi-convex w.r.t. the eigenvalue \(\lambda_j\)~\cite{lessard2016analysis,kim2018adaptive,wang2021asymptotic}.
Thus, the extremal value is attained on the boundary, i.e.
\begin{eqnarray}\label{eq:rhoT}
\rho(\boldsymbol T)=\max (|r^{\pm}(\lambda_{\max},\mu_t,\eta_t)|,|r^{\pm}(\lambda_{\min},\mu_t,\eta_t)|).
\end{eqnarray}
As a whole, \(\rho(\boldsymbol T)\) is determined by the maximum modulus of the roots of (\ref{eq:rhoT}).
We denote that surfaces \(\Pi_1\) and \(\Pi_2\) correspond to \(\lambda_{\min}\) and \(\lambda_{\max}\), respectively.

Below we show how to determine the minimum spectral radius and corresponding parameters.
Back to (\ref{eq:solution-eigen}), \(|r^{\pm}(\lambda_j,\mu_t,\eta_t)|\geq |(1+\eta_t)(1-\mu_t\lambda_j)|/2\) takes the equal if and only if \(\Delta (\lambda_j,\mu_t,\eta_t)=0\).
In this case, we can get a relationship of the parameter \((\mu_t,\eta_t)\)
\begin{eqnarray}\label{eq:opt-eta}
\eta_t^-=\frac{1-\sqrt{\mu_t\lambda_j}}{1+\sqrt{\mu_t\lambda_j}},\eta_t^+=\frac{1+\sqrt{\mu_t\lambda_j}}{1-\sqrt{\mu_t\lambda_j}}.
\end{eqnarray}
Obviously, \(0<\eta_t^-<1<\eta_t^+\).
Given $\mu_t$, there are three cases for $\eta_t$.
\begin{itemize}
\item
(\ref{eq:solution-eigen}) with \(\eta_t\in(0,\eta_{t}^-)\cup(\eta_{t}^+,\infty)\) has two different solutions.
\item
(\ref{eq:solution-eigen}) with \(\eta_t=\eta_{t}^\pm\) has a single solution.
\item
(\ref{eq:solution-eigen}) with \(\eta_t\in(\eta_{t}^-,\eta_{t}^+)\) has conjugate complex solutions.
\end{itemize}

If \(\eta_t\in[\eta_{t}^+,\infty)\), \(r^{\pm}(\lambda_j,\mu_t,\eta_t^+)\geq|(1+\eta_t^+)(1-\mu_t\lambda_j)|/2=1+\sqrt{\mu_t\lambda_j}>1\), and \(\rho(\boldsymbol T)>1\) is obtained form (\ref{eq:rhoT}).
Conversely, when \(\eta_t=\eta_{t}^-\), \(r^{\pm}(\lambda_j,\mu_t,\eta_t^-)=|(1+\eta_t^-)(1-\mu_t\lambda_j)|/2=1-\sqrt{\mu_t\lambda_j}<1\).
This is also why the parameter is selected as \(0<\eta\leq 1\) in practice.
When \(\eta_t\in(\eta_{t}^-,\eta_{t}^+)\), \(\rho(\boldsymbol T)=\max_{\lambda_j} \sqrt{\eta_t(1-\mu_t\lambda_j)}\) monotonically increases w.r.t. \(\eta_t\) and monotonically decreases w.r.t. \(\mu_t\).
We can draw the geometric properties of \(\rho(\boldsymbol T)\) w.r.t. \((\mu_t,\eta_t)\), and condition \(\Delta (\lambda_j,\mu_t,\eta_t^-)=0\) helps to find the theoretical lower bound of \(\rho(\boldsymbol T)\).
The optimal parameter pair \((\mu_\flat,\eta_\flat)\) is the intersection of $r^{\pm}(\lambda_{\min},\mu_\flat,\eta_\flat)$ in the curve \(\eta_t^-=\frac{1-\sqrt{\mu_t\lambda_j}}{1+\sqrt{\mu_t\lambda_j}}\) and the surface \(\Pi_2\), i.e., $|r^{-}(\lambda_{\max},\mu_t,\eta_t)|$.
So it satisfies the following equation
\begin{eqnarray*}
(1+\eta_\flat)(1-\mu_\flat\lambda_{\min})=-(1+\eta_\flat)(1-\mu_\flat\lambda_{\max}^2)+\sqrt{(1+\eta_\flat)^2(1-\mu_\flat\lambda_{\max}^2)^2-4\eta_\flat (1-\mu_\flat\lambda_{\max}^2)}.
\end{eqnarray*}
Bringing in \(\eta_\flat=\frac{1-\sqrt{\mu_\flat\lambda_{\min}}}{1+\sqrt{\mu_\flat\lambda_{\min}}}\), it is not difficult for us to get optimal convergence result $\mu_{\flat}=\frac{4}{\lambda_{\min}+3\lambda_{\max}}$ and $\rho_{\mathsf{opt}}(\boldsymbol T)=1-\sqrt{\frac{4\lambda_{\min}}{\lambda_{\min}+3\lambda_{\max}}}$ in (\ref{eq:NAG-opt-convergence}).
Also, for \(\eta_t<1\), the intersection of \(\Pi_1\) and \(\Pi_2\) can be calculated according to monotonicity
\begin{eqnarray*}
r^{+}(\lambda_{\min},\mu_t,\eta_t)=-r^{-}(\lambda_{\max},\mu_t,\eta_t).
\end{eqnarray*}
If \(\eta_t=0\), it simplifies to \(\mu_t=2/(\lambda_{\min}+\lambda_{\max})=\mu_\dagger\) in (\ref{eq:IHT-step}).
Due to momentum, the optimal stepsizes satisfy \(\mu_\flat<\mu_\dagger\).
In fact, we bring \(\eta_t=0\) to get \(\boldsymbol e_t=\boldsymbol H\boldsymbol e_{t-1}+\mathcal O(\|\boldsymbol e_{t-1}\|_2^2)\), which is consistent with the non-accelerated iteration.
Conversely, if \(\mu_t\geq \mu_\dagger\), then \(\eta_t=0\) is a good parameter choice, which means NAG degenerates to Grad.
When \(\eta_t\neq0\), we have
\begin{eqnarray*}
\eta_t\mu_t^2(\lambda_{\max}-\lambda_{\min})^2+2(1+\eta_t)^2(1-\mu_t\lambda_{\max})(1-\mu_t\lambda_{\min})(2-\mu_t(\lambda_{\min}+\lambda_{\max}))=0.
\end{eqnarray*}
Despite the complex form, we use symbolic computing tools to solve  when \(\mu_t\in(\mu_\flat,\mu_\dagger)\)
\begin{tiny}
\[\begin{aligned}
\eta_{t\Join}&=[(-4\lambda_{\min}^2\lambda_{\max}\mu_t^3+5\lambda_{\min}^2\mu_t^2-4\lambda_{\min}\lambda_{\max}^2\mu_t^3+14\lambda_{\min}\lambda_{\max}\mu_t^2-12\lambda_{\min}\mu_t+5\lambda_{\max}^2\mu_t^2-12\lambda_{\max}\mu_t+8)\\
&-\sqrt{\mu_t^2(-(\lambda_{\max}-\lambda_{\min})^2)(8\lambda_{\min}^2\lambda_{\max}\mu_t^3-9\lambda_{\min}^2\mu_t^2+8\lambda_{\min}\lambda_{\max}^2\mu_t^3-30\lambda_{\min}\mu_t^2+24\lambda_{\min}\mu_t-9\lambda_{\max}^2\mu_t^2+24\lambda_{\max}\mu_t-16)}]\\
&/(4(\lambda_{\min}\mu_t-1)(\lambda_{\max}\mu_t-1)(\lambda_{\min}\mu_t+\lambda_{\max}\mu_t-2)).
\end{aligned}\]
\end{tiny}
We analyze the spectral radius of \(\boldsymbol T\) in (\ref{eq:rhoT}) w.r.t. pair \((\mu_t,\eta_{t})\) by case.
\end{proof}

\section{Proof in Sect.~\ref{sec:convergence-Riemannian-optimization}}
\subsection{Proof of Lemma~\ref{lem:retraction-Perturbation}}
\label{app:lem:retraction-Perturbation}
\begin{proof}
The first one obviously holds according to Lemma~\ref{lem:TSVD-Perturbation}.
From (\ref{eq:orth-retraction}), we have
\begin{eqnarray*}
\begin{aligned}
\mathcal R_{\boldsymbol X}^{\mathsf{orth}}(\boldsymbol N)
&=(\boldsymbol X+\boldsymbol N)\boldsymbol V_{\boldsymbol X}(\boldsymbol \Sigma_{\boldsymbol X}+\boldsymbol U_{\boldsymbol X}^\top \boldsymbol N \boldsymbol V_{\boldsymbol X})^{-1}\boldsymbol U_{\boldsymbol X}^\top(\boldsymbol X+\boldsymbol N)\\
&\stackrel{(a)}{=}(\boldsymbol X+\boldsymbol N)\boldsymbol V_{\boldsymbol X}(\boldsymbol \Sigma_{\boldsymbol X}^{-1}-\boldsymbol \Sigma_{\boldsymbol X}^{-1}\boldsymbol U_{\boldsymbol X}^\top \boldsymbol N\boldsymbol V_{\boldsymbol X}\boldsymbol \Sigma_{\boldsymbol X}^{-1})\boldsymbol U_{\boldsymbol X}^\top(\boldsymbol X+\boldsymbol N)+\mathcal O(\|\boldsymbol N\|_F^2)\\
&=(\boldsymbol X+\boldsymbol N)(\boldsymbol V_{\boldsymbol X}\boldsymbol \Sigma_{\boldsymbol X}^{-1}\boldsymbol U_{\boldsymbol X}^\top-\boldsymbol V_{\boldsymbol X}\boldsymbol \Sigma_{\boldsymbol X}^{-1}\boldsymbol U_{\boldsymbol X}^\top \boldsymbol N \boldsymbol V_{\boldsymbol X}\boldsymbol \Sigma_{\boldsymbol X}^{-1}\boldsymbol U_{\boldsymbol X}^\top)(\boldsymbol X+\boldsymbol N)+\mathcal O(\|\boldsymbol N\|_F^2)\\
&=(\boldsymbol X+\boldsymbol N)(\boldsymbol X^{-\top}-\boldsymbol X^{-\top}\boldsymbol N \boldsymbol X^{-\top})(\boldsymbol X+\boldsymbol N)+\mathcal O(\|\boldsymbol N\|_F^2)\\
&\stackrel{(b)}{=}\boldsymbol X\boldsymbol X^{-\top}\boldsymbol X+\boldsymbol N\boldsymbol X^{-\top}\boldsymbol X+\boldsymbol X\boldsymbol X^{-\top}\boldsymbol N-\boldsymbol X\boldsymbol X^{-\top}\boldsymbol N \boldsymbol X^{-\top}\boldsymbol X+\mathcal O(\|\boldsymbol N\|_F^2)\\
&\stackrel{(c)}{=}\boldsymbol X+P_{\boldsymbol U_{\boldsymbol X}} \boldsymbol N+\boldsymbol N P_{\boldsymbol V_{\boldsymbol X}}-P_{\boldsymbol U_{\boldsymbol X}} \boldsymbol N P_{\boldsymbol V_{\boldsymbol X}}+\mathcal O(\|\boldsymbol N\|_F^2)\\
&=\boldsymbol X+\boldsymbol N-P_{\boldsymbol U_{\boldsymbol X}}^\perp \boldsymbol N P_{\boldsymbol V_{\boldsymbol X}}^\perp+\mathcal O(\|\boldsymbol N\|_F^2)\\
&=\mathcal P_{\mathbb T_{\boldsymbol X}}(\boldsymbol X+\boldsymbol N)+\mathcal O(\|\boldsymbol N\|_F^2),
\end{aligned}
\end{eqnarray*}
where $(a)$ is the perturbation analysis of matrix inverse.
As long as $\|\boldsymbol A^{-1}\boldsymbol B\|<1$ or $\|\boldsymbol B\boldsymbol A^{-1}\|<1$ holds, the Taylor expansion of the inverse of the matrix sum is as follows
\begin{eqnarray*}
\begin{aligned}
(\boldsymbol A+\boldsymbol B)^{-1}
&=\boldsymbol A^{-1} - \boldsymbol A^{-1}\boldsymbol B\boldsymbol A^{-1} + \boldsymbol A^{-1}(\boldsymbol B\boldsymbol A^{-1})^2 - \boldsymbol A^{-1}(\boldsymbol B\boldsymbol A^{-1})^3 + \cdots\\
&=\boldsymbol A^{-1}-\boldsymbol A^{-1}\boldsymbol B\boldsymbol A^{-1}+\mathcal O(\|\boldsymbol B\|_F^2).
\end{aligned}
\end{eqnarray*}
Using the norm inequality $\|\boldsymbol A\boldsymbol B\|\leq\|\boldsymbol A\| \|\boldsymbol B\|$, combined with the condition $\|\boldsymbol N\|\leq\|\boldsymbol N\|_F< \sigma_{\replaced{r}{\min}}(\boldsymbol X)/2$, it can be judged that the inverse matrix condition holds.
\begin{eqnarray*}
\|\boldsymbol \Sigma_{\boldsymbol X}^{-1}(\boldsymbol U_{\boldsymbol X}^\top \boldsymbol N \boldsymbol V_{\boldsymbol X})\|\leq\frac{\|\boldsymbol U_{\boldsymbol X}^\top \boldsymbol N \boldsymbol V_{\boldsymbol X}\|}{\|\boldsymbol \Sigma_{\boldsymbol X}\|}\leq\frac{\|\boldsymbol N\|}{\sigma_{\replaced{r}{\min}}(\boldsymbol X)}<1.
\end{eqnarray*}
$(b)$ merges the product of multiple $\boldsymbol N$ into higher-order terms.
$(c)$ uses the SVD of $\boldsymbol X$ to get
\begin{eqnarray*}
\begin{aligned}
&\boldsymbol X\boldsymbol X^{-\top}=\boldsymbol U_{\boldsymbol X}\boldsymbol \Sigma_{\boldsymbol X} \boldsymbol V_{\boldsymbol X}^\top \boldsymbol V_{\boldsymbol X}\boldsymbol \Sigma_{\boldsymbol X}^{-1}\boldsymbol U_{\boldsymbol X}^\top =\boldsymbol U_{\boldsymbol X}\boldsymbol U_{\boldsymbol X}^\top=P_{\boldsymbol U_{\boldsymbol X}},\\
&\boldsymbol X^{-\top}\boldsymbol X=\boldsymbol V_{\boldsymbol X}\boldsymbol \Sigma_{\boldsymbol X}^{-1}\boldsymbol U_{\boldsymbol X}^\top \boldsymbol U_{\boldsymbol X}\boldsymbol \Sigma_{\boldsymbol X} \boldsymbol V_{\boldsymbol X}^\top =\boldsymbol V_{\boldsymbol X}\boldsymbol V_{\boldsymbol X}^\top=P_{\boldsymbol V_{\boldsymbol X}},\\
&\boldsymbol X\boldsymbol X^{-\top}\boldsymbol X=P_{\boldsymbol U_{\boldsymbol X}} \boldsymbol U_{\boldsymbol X}\boldsymbol \Sigma_{\boldsymbol X} \boldsymbol V_{\boldsymbol X}^\top=\boldsymbol U_{\boldsymbol X}\boldsymbol \Sigma_{\boldsymbol X} \boldsymbol V_{\boldsymbol X}^\top=\boldsymbol X.
\end{aligned}
\end{eqnarray*}
\end{proof}

\subsection{Convergence for Algorithm~\ref{Alg-RGrad}}
\label{app:col:MatRGrad}

\begin{proof}
According to Algorithm~\ref{Alg-RGrad}, we have
\begin{eqnarray*}
\begin{aligned}
\boldsymbol E_{t+1}&=\mathcal R_{\boldsymbol X_t}(-\mu_t\text{grad}f(\boldsymbol X_t))-\boldsymbol X_\star\\
&\stackrel{(a)}{=}\mathcal P_{\mathbb T_{X_t}\mathbb M_r}(\boldsymbol X_t-\mu_t\nabla f(\boldsymbol X_t))-\boldsymbol X_\star+\mathcal O(\|\boldsymbol E_{t}\|_F^2)\\
&\stackrel{(b)}{=}(\boldsymbol E_{t}-\mu_t\nabla f(\boldsymbol X_t))-P_{\boldsymbol U_t}^\perp(\boldsymbol E_{t}-\mu_t\nabla f(\boldsymbol X_t))P_{\boldsymbol V_t}^\perp+\mathcal O(\|\boldsymbol E_{t}\|_F^2)\\
&\stackrel{(c)}{=}(\boldsymbol E_{t}-\mu_t\nabla f(\boldsymbol X_t))-P_{\boldsymbol U_\star}^\perp(\boldsymbol E_{t}-\mu_t\nabla f(\boldsymbol X_t))P_{\boldsymbol V_\star}^\perp+\mathcal O(\|\boldsymbol E_{t}\|_F^2),\\
\end{aligned}
\end{eqnarray*}
where \((a)\) uses Lemma~\ref{lem:retraction-Perturbation},
\((b)\) is based on the tangent space projection in (\ref{eq:tangent-space-projection}),
and \((c)\) uses the subspace perturbation in Lemma~\ref{lem:Wedin}, and replaces the subspace $\mathcal P_{\mathbb T_{\boldsymbol X_t}\mathbb M_r}$ with $\mathcal P_{\mathbb T_{\boldsymbol X_\star}\mathbb M_r}$.
\begin{eqnarray*}
\begin{aligned}
\|P_{\boldsymbol U_t}^\perp \boldsymbol AP_{\boldsymbol V_t}^\perp-P_{\boldsymbol U_\star}^\perp \boldsymbol AP_{\boldsymbol V_\star}^\perp\|
&=\|P_{\boldsymbol U_t}^\perp \boldsymbol AP_{\boldsymbol V_t}^\perp-P_{\boldsymbol U_t}^\perp \boldsymbol AP_{\boldsymbol V_\star}^\perp+P_{\boldsymbol U_t}^\perp \boldsymbol AP_{\boldsymbol V_\star}^\perp-P_{\boldsymbol U_\star}^\perp \boldsymbol AP_{\boldsymbol V_\star}^\perp\|\\
&\leq \|P_{\boldsymbol U_t}^\perp \boldsymbol AP_{\boldsymbol V_t}^\perp-P_{\boldsymbol U_t}^\perp \boldsymbol AP_{\boldsymbol V_\star}^\perp\|+\|P_{\boldsymbol U_t}^\perp \boldsymbol AP_{\boldsymbol V_\star}^\perp-P_{\boldsymbol U_\star}^\perp \boldsymbol AP_{\boldsymbol V_\star}^\perp\|\\
&\leq \|P_{\boldsymbol U_t}^\perp\| \|\boldsymbol A\| \|P_{\boldsymbol V_t}^\perp-P_{\boldsymbol V_\star}^\perp\|+\|P_{\boldsymbol U_t}^\perp-P_{\boldsymbol U_\star}^\perp\| \|\boldsymbol A\| \|P_{\boldsymbol V_\star}^\perp\|\\
&=\mathcal O(\|\boldsymbol E_t\|_F^2),
\end{aligned}
\end{eqnarray*}
The subsequent proof is consistent with the proof of Theorem~\ref{th:MatIHT} in Appendix~\ref{app:Th:MatIHT}.
\end{proof}

\subsection{Convergence for Algorithm~\ref{Alg-NARG}}
\label{app:col:MatNARG}
\begin{proof}
The proof is divided into three steps to analyse $\boldsymbol X_{t-1}$, $\boldsymbol Y_t$ and $\boldsymbol X_{t+1}$, respectively.

Step 1: Calculate the orthographic retraction of $\boldsymbol X_{t-1}$ and the inverse matrix.
\begin{eqnarray*}
\begin{aligned}
\mathsf{inv} \mathcal R^{\mathsf {orth}}_{\boldsymbol X_t}(\boldsymbol X_{t-1})
&=\mathcal P_{\mathbb T_{\boldsymbol X_t}\mathbb M_r}(\boldsymbol X_{t-1}-\boldsymbol X_t)\\
&=\mathcal P_{\mathbb T_{\boldsymbol X_t}\mathbb M_r}(\boldsymbol X_{t-1})-\boldsymbol X_t\\
&=\mathcal P_{\mathbb T_{\boldsymbol X_\star}\mathbb M_r}(\boldsymbol X_{t-1})-\boldsymbol X_t+\mathcal O(\|\boldsymbol E_{t}\|_F^2)\\
&=\boldsymbol X_{t-1}-\boldsymbol X_t+\mathcal O(\|\boldsymbol E_{t}\|_F^2+\|\boldsymbol E_{t-1}\|_F^2).\\
\end{aligned}
\end{eqnarray*}
It gives an approximation of $\boldsymbol X_{t-1}$ on the tangent space $\mathbb T_{\boldsymbol X_t}\mathbb M_r$.

Step 2: Similar to Appendix~\ref{app:th:MatNAG}, we calculate the residual of $\boldsymbol Y_{t}$
\begin{eqnarray*}
\begin{aligned}
\boldsymbol Y_{t}-\boldsymbol X_\star&=\mathcal R^{\mathsf {orth}}_{\boldsymbol X_t}(-\eta_t \mathsf{inv} \mathcal R^{\mathsf {orth}}_{\boldsymbol X_t}(\boldsymbol X_{t-1}))-X_\star\\
&=\mathcal P_{\mathbb T_{\boldsymbol X_t}\mathbb M_r}(\boldsymbol X_t-\eta_t \mathsf{inv} \mathcal R^{\mathsf {orth}}_{\boldsymbol X_t}(\boldsymbol X_{t-1}))-\boldsymbol X_\star+\mathcal O(\|\boldsymbol E_t\|_F^2)\\
&=\boldsymbol X_t-\eta_t \mathsf{inv} \mathcal R^{\mathsf {orth}}_{\boldsymbol X_t}(\boldsymbol X_{t-1})-\boldsymbol X_\star+\mathcal O(\|\boldsymbol E_t\|_F^2)\\
&=\boldsymbol X_t-\boldsymbol X_\star+\eta_t (\boldsymbol X_t-\boldsymbol X_{t-1})+\mathcal O(\|\boldsymbol E_{t}\|_F^2+\|\boldsymbol E_{t-1}\|_F^2)\\
&=\boldsymbol E_t+\eta_t(\boldsymbol E_t-\boldsymbol E_{t-1})+\mathcal O(\|\boldsymbol E_{t}\|_F^2+\|\boldsymbol E_{t-1}\|_F^2).
\end{aligned}
\end{eqnarray*}
It also satisfies the linear extrapolation in Euclidean space.

Step 3: Compute $\boldsymbol X_{t+1}-\boldsymbol X_\star$ to get the recursive form
\begin{eqnarray*}
\begin{aligned}
\boldsymbol E_{t+1}&=\boldsymbol X_{t+1}-\boldsymbol X_\star\\
&=\mathcal R^{\mathsf {orth}}_{\boldsymbol Y_t}(-\mu_t\text{grad}f(\boldsymbol Y_t))-\boldsymbol X_\star\\
&=\mathcal P_{\mathbb T_{\boldsymbol Y_t}\mathbb M_r}(\boldsymbol Y_t-\mu_t \nabla f(\boldsymbol Y_t))-\boldsymbol X_\star+\mathcal O(\|\boldsymbol Y_t-\boldsymbol X_\star\|_F^2)\\
&=\mathcal P_{\mathbb T_{\boldsymbol X_\star}\mathbb M_r}(\boldsymbol Y_t-\mu_t \nabla f(\boldsymbol Y_t))-\boldsymbol X_\star+\mathcal O(\|\boldsymbol Y_t-\boldsymbol X_\star\|_F^2)\\
&=(\boldsymbol Y_t-\boldsymbol X_\star-\mu_t \nabla f(\boldsymbol Y_t))-P_{\boldsymbol U_\star}^\perp (\boldsymbol Y_t-\boldsymbol X_\star-\mu_t \nabla f(\boldsymbol Y_t))P_{\boldsymbol V_\star}^\perp +\mathcal O(\|\boldsymbol Y_t-\boldsymbol X_\star\|_F^2).\\
\end{aligned}
\end{eqnarray*}
The subsequent proof is consistent with proof of Theorem~\ref{th:MatNAG} in Appendix~\ref{app:th:MatNAG}.
\end{proof}

\subsection{Proof of Restart Condition Equivalence in (\ref{eq:restart-condition})}
\label{app:col:Restart}
\begin{proof}
When condition (\ref{eq:Basin-of-Attraction}) hold,  we have
\begin{eqnarray*}
\begin{aligned}
\begin{aligned}
\langle \nabla f(\boldsymbol Y_{t-1}), \boldsymbol X_t-\boldsymbol X_{t-1}\rangle
&=\langle \text{grad}~f(\boldsymbol Y_{t-1})+\nabla f(\boldsymbol Y_{t-1})-\text{grad}~f(\boldsymbol Y_{t-1}), X_t-\boldsymbol X_{t-1}\rangle\\
&\stackrel{(a)}{=}\langle \text{grad}~f(\boldsymbol Y_{t-1}), \mathsf{inv} \mathcal R^{\mathsf {orth}}_{\boldsymbol Y_{t-1}}(\boldsymbol X_{t})-\mathsf{inv} \mathcal R^{\mathsf {orth}}_{\boldsymbol Y_{t-1}}(\boldsymbol X_{t-1})\rangle\\
&\quad+\langle \nabla f(\boldsymbol Y_{t-1})-\text{grad}~f(\boldsymbol Y_t), \boldsymbol X_t-\mathsf{inv} \mathcal R^{\mathsf {orth}}_{\boldsymbol Y_{t-1}}(\boldsymbol X_{t})\rangle\\
&\quad-\langle \nabla f(\boldsymbol Y_{t-1})-\text{grad}~f(\boldsymbol Y_t), \boldsymbol X_{t-1}-\mathsf{inv} \mathcal R^{\mathsf {orth}}_{\boldsymbol Y_{t-1}}(\boldsymbol X_{t-1})\rangle\\
&\stackrel{(b)}{\approx}\langle \text{grad}~f(\boldsymbol Y_{t-1}), \mathsf{inv} \mathcal R^{\mathsf {orth}}_{\boldsymbol Y_{t-1}}(\boldsymbol X_{t})-\mathsf{inv} \mathcal R^{\mathsf {orth}}_{\boldsymbol Y_{t-1}}(\boldsymbol X_{t-1})\rangle,\\
\end{aligned}
\end{aligned}
\end{eqnarray*}
where $(a)$ uses the orthogonal relationship of the Riemannian gradient and tangent space.
Based on the first-order expansion, we appropriately omit the higher-order terms in $(a)$ to obtain the approximate relationship $(b)$, which will not change the sign before and after the approximation.
As mentioned in step 2 in Appendix~\ref{app:col:MatNARG}, $\nabla f(\boldsymbol Y_{t-1})=\boldsymbol \Theta \text{vec}(\boldsymbol Y_{t-1}-\boldsymbol X_\star)$ and $\text{grad}~f(\boldsymbol Y_{t-1})$ both are first order w.r.t. the residual.
According to Lemma~\ref{lem:retraction-Perturbation}, $\boldsymbol X_t-\mathsf{inv} \mathcal R^{\mathsf {orth}}_{\boldsymbol Y_{t-1}}(\boldsymbol X_{t})$ and $\boldsymbol X_{t-1}-\mathsf{inv} \mathcal R^{\mathsf {orth}}_{\boldsymbol Y_{t-1}}(\boldsymbol X_{t-1})$ are second order.
So the last two terms of $(a)$ are third order, while the remaining inner product is second order.
\end{proof}

\bibliographystyle{spmpsci}      

\end{document}